\documentclass[11pt]{amsart}
\usepackage[margin=1.15in]{geometry}
\usepackage{amsmath,amssymb,amsthm,mathtools}
\usepackage{enumitem}
\usepackage[colorlinks=true,linkcolor=blue,citecolor=blue,urlcolor=blue]{hyperref}

\numberwithin{equation}{section}
\newtheorem{theorem}{Theorem}[section]
\newtheorem{lemma}[theorem]{Lemma}
\newtheorem{proposition}[theorem]{Proposition}
\newtheorem{corollary}[theorem]{Corollary}
\theoremstyle{definition}
\newtheorem{definition}[theorem]{Definition}
\newtheorem{example}[theorem]{Example}
\theoremstyle{remark}
\newtheorem{remark}[theorem]{Remark}

\newcommand{\Fq}{\mathbb{F}_q}
\newcommand{\OK}{\mathcal{O}_K}
\newcommand{\N}{\mathcal{N}}
\newcommand{\kinf}{k_\infty}
\newcommand{\Z}{\mathbb{Z}}
\newcommand{\R}{\mathbb{R}}
\newcommand{\T}{\mathbb{T}}
\newcommand{\Hk}{\chi_{\mathrm{Hecke}}}
\newcommand{\f}{\mathfrak{f}}
\newcommand{\p}{\mathfrak{p}}
\newcommand{\aid}{\mathfrak{a}}
\newcommand{\Dd}{\mathfrak{D}}
\newcommand{\cS}{\mathcal{S}}
\DeclareMathOperator{\sgn}{sgn}
\DeclareMathOperator{\Norm}{Norm}
\DeclareMathOperator{\Cl}{Cl}
\DeclareMathOperator{\Pic}{Pic}
\DeclareMathOperator{\supp}{supp}
\renewcommand{\deg}{\operatorname{deg}}
\newcommand{\pu}[1]{\langle #1\rangle}
\newcommand{\eqs}{\equiv^{*}}
\newcommand{\e}{\mathrm{e}}

\begin{document}

\title{The distribution of $\alpha p$ modulo one in real quadratic function fields}
\author{Esrafil Ali Molla}
\address{Department of Mathematics, Indian Institute of Technology Bombay, Mumbai 400076, India}
\email{esrafil.math@gmail.com}

\begin{abstract}
In the thirties of the last century, I.\,M. Vinogradov proved that the inequality $\|p\alpha\|\le p^{-1/5+\varepsilon}$ has infinitely many prime solutions $p$, where $\|\cdot\|$ denotes the distance to a nearest integer. This result has subsequently been improved by many authors. In particular, Vaughan (1978) replaced the exponent $1/5$ by $1/4$, using his celebrated identity for the von Mangoldt function and a refinement of Fourier analytic arguments. The current record is due to Matom\"aki (2009), who showed that the inequality $\|p\alpha\|\le p^{-1/3+\varepsilon}$ has infinitely many prime solutions. This exponent $1/3$ is considered the limit of current technology. Baier and the author established an analogue of Matom\"aki's result for the function field $k=\Fq(T)$ and its imaginary quadratic extensions. Moreover, for real quadratic extensions of $k$ of class number one, they obtained an analogue of Vaughan's result (exponent $1/4$). In this paper, we establish the analogue of Matom\"aki's result for all real quadratic extensions of $k$ without any restriction on the class number.
\end{abstract}

\maketitle

\section{Introduction and statements of the main results}

A fundamental result in Diophantine approximation is the following theorem due to Dirichlet.

\begin{theorem}[Dirichlet]
Given any real irrational $\alpha$, there are infinitely many pairs $(a,q)\in\Z\times\mathbb{N}$ of relatively prime integers such that
\[
\Big|\alpha-\frac{a}{q}\Big|<q^{-2}.
\]
\end{theorem}

While it is easy to prove the above statement using the pigeonhole principle or the continued fraction expansion of $\alpha$, understanding Diophantine approximation by fractions whose denominators are restricted to subsets of the positive integers, such as the set of primes, is generally challenging. The problem of Diophantine approximation with primes has a rich history and triggered the development of several important tools in analytic number theory. In this context, the question becomes: for which $\theta>0$ do there exist infinitely many primes $p$ such that
\begin{equation}\label{eq:classical}
\|p\alpha\|<p^{-\theta+\varepsilon},
\end{equation}
where $\|x\|$ is the distance of $x\in\R$ to the nearest integer? In the thirties of the last century, I.\,M. Vinogradov \cite{Vin} established the infinitude of primes $p$ satisfying \eqref{eq:classical} with $\theta=1/5$. Vaughan \cite{Vau} improved the exponent to $\theta=1/4$ by utilising his celebrated identity for the von Mangoldt function and refined Fourier analytic arguments. Harman \cite{Har83} developed a new sieve method which enabled him to obtain $\theta=3/10$. Jia and Harman made several improvements \cite{Jia1,Jia2,Har96}, and Heath-Brown and Jia \cite{HBJ} applied the Weil bound for Kloosterman sums to reach $\theta=16/49$, which is very close to $1/3$. Eventually, Matom\"aki \cite{Mat} achieved $\theta=1/3$ unconditionally by employing bounds for sums of Kloosterman sums. This is considered to be the limit of current technology. It was mentioned in \cite{HBJ} without proof that $\theta=1/3$ is not difficult to obtain under the Generalised Riemann Hypothesis (GRH) for Dirichlet $L$-functions. This proof was carried out in \cite{BaMo} by Baier and the author.

Exploring analogues of this problem in number fields is interesting, as novel ideas are needed to adapt the classical methods to this context. Baier \cite{Bai} extended the classical problem in a slightly generalised form to $\mathbb{Q}(i)$, and Harman \cite{Har19} established the full analogue of his exponent $\theta=7/22$ for $\mathbb{Q}(i)$. Utilising Harman's sieve method, Baier, Mazumder and Technau established versions of these results for quadratic number fields of class number one \cite{BaTe,BaMa,BMT}. Recently, Baier, Das and the author \cite{BDM} established the analogue of Matom\"aki's exponent $\theta=1/3$ for all quadratic number fields under GRH for their Hecke $L$-functions.

In the function field setting, the Riemann hypothesis for Hecke $L$-functions is known to hold due to Weil \cite{Wei48}, and therefore our result is unconditional. In \cite{BaMo}, Baier and the author established an analogue of Matom\"aki's result for the rational function field $k=\Fq(T)$ and its imaginary quadratic extensions, using the Riemann hypothesis for Hecke $L$-functions applied to Hecke characters of trivial infinity type. This is possible since the unit group of $\OK$ is finite in the imaginary quadratic case. The arguments in \cite{BaMo} do not easily carry over to real quadratic extensions $K$ of $k$ without introducing new ideas. The main obstacle is the presence of an infinite group of units in the integral closure $\OK$ of $\Fq[T]$ in $K$: for an ideal $\f$ of $\OK$, the image of $\OK^\times$ in $(\OK/\f)^\times$ can become very large, so that the Hecke characters of trivial infinity type used in \cite{BaMo} no longer suffice to detect residue classes. For this reason, in \cite{BaMo2}, Baier and the author used Vaughan's identity and exponential sums instead of $L$-functions, and obtained a result corresponding to Vaughan's exponent $\theta=1/4$ for real quadratic function fields of class number one.

In this article, we establish the analogue of Matom\"aki's result (exponent $\theta=1/3$) for all real quadratic function fields, without any restriction on the class number. Together with \cite{BaMo}, the analogue of Matom\"aki's theorem thus holds, for odd $q$, for $\Fq(T)$ and for all of its quadratic extensions. Our method is inspired by that in \cite{BDM} in the case of real quadratic number fields. To handle the above-mentioned obstruction coming from the infinitely many units, we use Hecke characters of non-trivial infinity type, which we describe explicitly for quadratic function fields. A pleasant new feature of the function field case is that the relevant infinity types form a finite set, so that the linear combinations of Hecke characters arising from Fourier analysis are exact (see the description of our method below). Below are the notations used throughout this paper, followed by the statement of our main result. This section concludes with a brief description of our method.

\subsection{General notations}
\begin{itemize}
\item Throughout, $q$ is a prime power, $A=\Fq[T]$, $k=\Fq(T)$ and $\kinf=\Fq((1/T))$ is the completion of $k$ at the place $\infty$, equipped with the absolute value $|x|=q^{\deg x}$ (see Section~\ref{sec:prelim}).
\item For a finite separable extension $K$ of $k$, we denote by $\OK$ the integral closure of $A$ in $K$, by $\OK^\times$ its group of units, and by $h_K=\#\Cl(\OK)$ the (ideal) class number.
\item If $\aid$ is a non-zero ideal of $\OK$, then $\N(\aid)=\#(\OK/\aid)$ denotes its norm. If $c\in\OK\setminus\{0\}$, we write $\N(c)=\N(c\OK)$.
\item Following usual convention, $\varepsilon$ denotes an arbitrarily small but fixed positive real number.
\item All $O$-constants and implied constants in $\ll$ are allowed to depend on $K$ and $\varepsilon$, unless stated otherwise. In Sections~\ref{sec:dirichlet} and~\ref{sec:proof}, a point $x\in\kinf^2$ is fixed, and statements of the form ``for $F$ large enough'' allow the threshold to depend on $x$ as well.
\end{itemize}

\subsection{Real quadratic function fields}
Now let $q$ be odd and $K=k(\sqrt D)$ a real quadratic extension of $k$, i.e.\ $D\in A$ is square-free of even degree with a square leading coefficient; without loss of generality $D$ is monic. In this context we use the following notations.
\begin{itemize}
\item We fix $\sqrt D\in\kinf$ and denote the two embeddings of $K$ into $\kinf$, the identity and conjugation, by
\[
\sigma_1(a+b\sqrt D):=a+b\sqrt D\quad\text{and}\quad\sigma_2(a+b\sqrt D):=a-b\sqrt D\qquad(a,b\in k).
\]
\item We write $\sigma(K):=\{(\sigma_1(x),\sigma_2(x)):x\in K\}\subset\kinf^2$.
\end{itemize}

We shall establish the following.

\begin{theorem}\label{thm:main}
Let $q$ be an odd prime power and $K$ a real quadratic extension of $k=\Fq(T)$ with arbitrary class number $h_K$. Then for any pair $(x_1,x_2)\in\kinf^2\setminus\sigma(K)$, there exist infinitely many non-zero principal prime ideals $\p$ in $\OK$ such that
\begin{equation}\label{eq:main}
\Big|x_i-\frac{\sigma_i(b)}{\sigma_i(\pi)}\Big|\le\N(\p)^{-2/3+\varepsilon}\qquad(i=1,2)
\end{equation}
for some generator $\pi$ of $\p$ and $b\in\OK$.
\end{theorem}

The above Theorem~\ref{thm:main} should be compared to the version of the Dirichlet approximation theorem for real quadratic function fields in Theorem~\ref{thm:dirichlet} below, which we will use in our proof and which gives the exponent $-1$ in place of $-2/3+\varepsilon$ without any restriction on the denominator. For suitably normalised generators $\pi$ we have $|\sigma_1(\pi)|\asymp|\sigma_2(\pi)|\asymp\N(\p)^{1/2}$, so \eqref{eq:main} is the precise analogue of Matom\"aki's exponent $\theta=1/3$: the bound $\N(\p)^{-2/3}=\N(\p)^{-(1+\theta)/2}$ corresponds to $|\alpha-b/p|\le p^{-1-\theta}$ in the classical setting. Theorem~\ref{thm:main} is the function field analogue of the main result of \cite{BDM} for real quadratic number fields, which, however, is conditional on GRH, whereas our result is unconditional.

\begin{remark}
For comparison, the main result of \cite{BaMo2} gives, for real quadratic function fields of class number one and $q>2^{12}$, the exponent $-\nu+\log_q\sqrt2+\varepsilon$ with $\nu=1/2+1/8=5/8$ in place of $-2/3+\varepsilon$, which is a weaker result since $5/8<2/3$. Theorem~\ref{thm:main} removes the restrictions on the class number and on $q$ (apart from $q$ being odd), and obtains the optimal exponent of current technology.
\end{remark}

\subsection{Description of our method}
To establish Theorem~\ref{thm:main}, we first approximate $(x_1,x_2)$ using the Dirichlet approximation theorem (Theorem~\ref{thm:dirichlet}), which we prove in a self-contained way for all real quadratic function fields. This is followed by a reduction of our problem to counting prime elements in unions of certain multiplicative residue classes modulo an ideal $\f$. Their membership of residue classes is then detected using Dirichlet characters modulo $\f$, and the principality of the prime ideals is detected using class group characters. The main contribution comes from the principal characters. To evaluate this main term, we count principal prime ideals and integral ideals in a given class with norm restricted by suitable quantities. To bound the error term, we need to estimate sums of Dirichlet characters over prime elements. We relate them to sums of Hecke characters over prime ideals, which can be handled by the Riemann hypothesis for Hecke $L$-functions. In the number field case \cite{BDM}, this hypothesis (GRH) had to be assumed; in the function field case, it is known due to Weil \cite{Wei48} (see also \cite[Ch.~9]{Ros}), so Theorem~\ref{thm:main} is unconditional.

As in the case of real quadratic number fields \cite{BDM}, the latter step is where the infinitely many units create difficulties: a Dirichlet character $\chi$ modulo $\f$ is in general not constant on the generators of a principal ideal, and there are generally not enough Hecke characters of trivial infinity type. We therefore need a precise description of the Hecke characters of non-trivial infinity type, which we develop in Section~\ref{sec:hecke} for global function fields and work out explicitly for real quadratic function fields. For real quadratic number fields, it was necessary in \cite{BDM} to approximate the relevant factors $|\sigma_1(p)/\sigma_2(p)|^{i\theta}$ by truncated Fourier series in the Gr\"o{\ss}encharaktere $|\sigma_1(p)/\sigma_2(p)|^{i\pi n/\log\epsilon}$, which required a delicate treatment of the tail using the spacing of the numbers $\log|\sigma_1(p)/\sigma_2(p)|$. In the function field case, the situation is more favourable. We show (Corollary~\ref{cor:finite}) that, up to norm twists, the infinity types of Hecke characters with given finite part form a finite set, parametrised by $n\bmod R$, where $R$ is the regulator of $K$. Consequently, for a prime element $p$ normalised with respect to the unit group, $\chi(p)$ is an exact finite linear combination of values of Hecke characters at the ideal $(p)$, with coefficients coming from a discrete Fourier transform on $\Z/2R\Z$ (Proposition~\ref{prop:decomp}). No truncation error arises. The remaining counting problems are handled using the ultrametric structure of $\kinf$, which makes lattice point counts elementary.

\begin{remark}[Other ideal classes]\label{rem:classes}
Theorem~\ref{thm:main} concerns principal prime ideals. As in \cite{BDM}, one may fix a class $C\in\Cl(\OK)$ and an integral ideal $\aid_0$ in $C^{-1}$, and consider approximations with denominators generating $\p\aid_0$ with $\p\in C$. We expect that our method applies with only notational changes, the class group characters now being evaluated at $C$ instead of the trivial class, but we do not carry this out here.
\end{remark}

\subsection{Use of AI tools}
The author used Claude (Anthropic, Claude Opus 5.5) as a research and writing assistant. The research problem, the overall strategy and the decisions on scope are the author's. Based on the author's material and instructions, Claude helped to work out the explicit description of Hecke characters in Section~\ref{sec:hecke}, proposed the exact decomposition of Proposition~\ref{prop:decomp}, assisted with the proofs in Sections~\ref{sec:dirichlet}--\ref{sec:proof}, including the self-contained proof of Theorem~\ref{thm:dirichlet}, and produced a first draft of the text. The author has checked all statements and proofs and takes full responsibility for the content.

\subsection{Organisation}
In Section~\ref{sec:prelim} we collect the necessary background on $\kinf$ and on quadratic function fields. Section~\ref{sec:hecke} is devoted to Hecke characters of global function fields, with an explicit description for real quadratic fields, their $L$-functions, and a prime number theorem for Hecke characters. Section~\ref{sec:dirichlet} contains a self-contained proof of the Dirichlet approximation theorem and Section~\ref{sec:counting} the counting lemmas. Theorem~\ref{thm:main} is proved in Section~\ref{sec:proof}. In Subsection~\ref{subsec:consistency} we show that the Hecke characters used in \cite{BaMo} for imaginary quadratic function fields are consistent with our general description: they are exactly the Hecke characters of trivial infinity type.

\section{Notation and preliminaries}\label{sec:prelim}

\subsection{The completion at infinity}
Let $q$ be a prime power, $A=\Fq[T]$, $k=\Fq(T)$ and $\kinf=\Fq((1/T))$. Every $x\in\kinf^\times$ can be written uniquely as
\begin{equation}\label{eq:decomp}
x=\sgn(x)\,T^{\deg x}\pu{x},\qquad \sgn(x)\in\Fq^\times,\ \deg x\in\Z,\ \pu{x}\in U^{(1)}:=1+T^{-1}\Fq[[T^{-1}]],
\end{equation}
so that $\sgn(x)$ is the leading coefficient and $\deg x$ the degree of the Laurent series $x$. The maps $\sgn$, $\deg$ and $\pu{\cdot}$ are homomorphisms. We put $|x|=q^{\deg x}$ and $|0|=0$. This absolute value is ultrametric: $|x+y|\le\max(|x|,|y|)$, with equality if $|x|\neq|y|$. For $x\in\kinf$ let $\|x\|=\min_{a\in A}|x-a|$. More generally, for $r\ge1$ we use the sup-norm $|(x_1,\dots,x_r)|=\max_i|x_i|$ on $\kinf^r$.

\subsection{Quadratic function fields}
From now on, $q$ is odd. Let $D\in A$ be a non-constant square-free polynomial and $K=k(\sqrt D)$. Let $\OK$ be the integral closure of $A$ in $K$. Then
\[
\OK=A+A\sqrt D
\]
(see \cite[p.~248]{Ros}), $\Fq$ is the full constant field of $K$, and $\OK$ is a Dedekind domain whose non-zero prime ideals correspond to the places of $K$ not lying above $\infty$. We write $\Norm(a+b\sqrt D)=a^2-b^2D\in k$ for $a,b\in k$, and for a non-zero ideal $\aid$ of $\OK$ we put $\N(\aid)=\#(\OK/\aid)$. If $\aid=(c)$ with $c\in\OK\setminus\{0\}$, then $\N(\aid)=|\Norm(c)|$, and we also write $\N(c)=\N(\aid)$. The norm $\N$ extends multiplicatively to fractional ideals and to $K^\times$.

The behaviour of $\infty$ in $K$ depends on $D$ (see e.g.\ \cite[Chapter~14]{Ros}):
\begin{itemize}
\item real quadratic case: $\deg D$ is even and $\sgn(D)$ is a square in $\Fq^\times$; then $\infty$ splits into two places $\infty_1,\infty_2$ of degree one;
\item inert case: $\deg D$ is even and $\sgn(D)$ is a non-square; then $\infty$ is inert, with one place $\infty_K$ of degree two;
\item ramified case: $\deg D$ is odd; then $\infty$ is ramified, with one place $\infty_K$ of degree one.
\end{itemize}
The last two cases are called imaginary quadratic. In all cases, the genus $g$ of $K$ is given by $\deg D\in\{2g+1,2g+2\}$.

\subsection{Real quadratic function fields}\label{subsec:real}
Let $K=k(\sqrt D)$ be real quadratic; replacing $D$ by $c^{-2}D$, where $c^2=\sgn(D)$, we may and do assume that $D$ is monic, so $\deg D=2g+2$. By Hensel's lemma, $D$ has a square root in $\kinf$, which we fix and denote by $\sqrt D$; then $|\sqrt D|=q^{g+1}$. The two embeddings of $K$ into $\kinf$ are
\[
\sigma_1(a+b\sqrt D)=a+b\sqrt D,\qquad\sigma_2(a+b\sqrt D)=a-b\sqrt D\qquad(a,b\in k),
\]
and they correspond to the two places $\infty_1,\infty_2$. For $c\in K^\times$ we have $\sigma_1(c)\sigma_2(c)=\Norm(c)$, hence
\begin{equation}\label{eq:norm-sigma}
|\sigma_1(c)|\,|\sigma_2(c)|=\N(c),\qquad \deg\sigma_1(c)+\deg\sigma_2(c)=\log_q\N(c).
\end{equation}
By Dirichlet's unit theorem for $S$-units (see \cite[Proposition~14.2]{Ros}), with $S=\{\infty_1,\infty_2\}$, $\OK^\times=\Fq^\times\times\epsilon^{\Z}$ for a fundamental unit $\epsilon$. If $u\in\OK^\times$, then $\Norm(u)\in A^\times=\Fq^\times$, so $\deg\sigma_2(u)=-\deg\sigma_1(u)$. If $\deg\sigma_1(\epsilon)=0$, then $\epsilon$ would have no poles at all, hence would be a constant, which is absurd. Replacing $\epsilon$ by $\epsilon^{-1}$ and then by $\epsilon/\sgn(\sigma_1(\epsilon))$, we may therefore normalise $\epsilon$ by
\begin{equation}\label{eq:eps-norm}
\sgn(\sigma_1(\epsilon))=1,\qquad R:=\deg\sigma_1(\epsilon)>0.
\end{equation}
The positive integer $R$ is the regulator of $K$. Since $\sgn$, $\deg$ and $\pu{\cdot}$ are multiplicative and trivial resp.\ the identity on constants, we have
\begin{equation}\label{eq:eps-sigma2}
\deg\sigma_2(\epsilon)=-R,\qquad\sgn(\sigma_2(\epsilon))=\Norm(\epsilon)\in\Fq^\times,\qquad\pu{\sigma_2(\epsilon)}=\pu{\sigma_1(\epsilon)}^{-1}.
\end{equation}

The following lemma is the function field analogue of a lemma in \cite[\S4]{BDM}; its proof is simpler because all quantities involved are integers.

\begin{lemma}[Normalised generators]\label{lem:normgen}
Let $K$ be real quadratic. Every non-zero principal ideal $\aid$ of $\OK$ has a unique generator $a$ such that
\begin{equation}\label{eq:normgen}
\sgn(\sigma_1(a))=1\quad\text{and}\quad -R<\deg\sigma_1(a)-\deg\sigma_2(a)\le R.
\end{equation}
We call it the normalised generator of $\aid$. If $\N(\aid)=q^n$, then
\begin{equation}\label{eq:normgen-size}
q^{(n-R)/2}\le|\sigma_i(a)|\le q^{(n+R)/2}\qquad(i=1,2).
\end{equation}
\end{lemma}

\begin{proof}
Let $a_0$ be any generator. All generators are $c\epsilon^ja_0$ with $c\in\Fq^\times$, $j\in\Z$. By \eqref{eq:eps-norm} and \eqref{eq:eps-sigma2},
\[
\deg\sigma_1(c\epsilon^ja_0)-\deg\sigma_2(c\epsilon^ja_0)=\deg\sigma_1(a_0)-\deg\sigma_2(a_0)+2Rj,
\]
so there is exactly one $j$ for which this lies in $(-R,R]$, and $\sgn(\sigma_1(c\epsilon^ja_0))=c\sgn(\sigma_1(a_0))$ determines $c$ uniquely. For \eqref{eq:normgen-size}, put $d_i=\deg\sigma_i(a)$. Then $d_1+d_2=n$ by \eqref{eq:norm-sigma}, and $-R<d_1-d_2\le R$ gives $(n-R)/2<d_1\le(n+R)/2$ and $(n-R)/2\le d_2<(n+R)/2$.
\end{proof}

\begin{lemma}\label{lem:Vy}
For $y\in\Z$ let $V_y=\{\kappa\in\OK:|\sigma_1(\kappa)|\le q^y,\ |\sigma_2(\kappa)|\le q^y\}$. Then $V_y$ is an $\Fq$-vector space, and if $y\ge g+1$, then $\dim_{\Fq}V_y=2y+1-g$.
\end{lemma}

\begin{proof}
Write $\kappa=a+b\sqrt D$ with $a,b\in A$. Since $q$ is odd, $|2|=1$, and $2a=\sigma_1(\kappa)+\sigma_2(\kappa)$, $2b\sqrt D=\sigma_1(\kappa)-\sigma_2(\kappa)$. By the ultrametric inequality, $\kappa\in V_y$ if and only if $|a|\le q^y$ and $|b\sqrt D|\le q^y$, i.e.\ $\deg a\le y$ and $\deg b\le y-g-1$ (with the convention $\deg0=-\infty$). This describes an $\Fq$-vector space of dimension $(y+1)+(y-g)$.
\end{proof}

\subsection{Arithmetic functions}
For non-zero ideals of $\OK$ we use the M\"obius function $\mu$, Euler's function $\varphi(\aid)=\#(\OK/\aid)^\times$, the number of prime divisors $\omega(\aid)$ and the number of divisors $\tau(\aid)$. We have $\sum_{\mathfrak{d}\mid\aid}\mu(\mathfrak{d})\N(\mathfrak{d})^{-1}=\varphi(\aid)/\N(\aid)$.

\begin{lemma}\label{lem:arith}
Let $K$ be a global function field and $\OK$ as above. For every $\eta>0$ we have
\[
2^{\omega(\aid)}\le\tau(\aid)\ll_{\eta,K}\N(\aid)^{\eta}\quad\text{and}\quad\varphi(\aid)\gg_{\eta,K}\N(\aid)^{1-\eta}.
\]
\end{lemma}

\begin{proof}
Write $\aid=\prod\p^{e_\p}$. We have $\tau(\aid)\N(\aid)^{-\eta}=\prod_\p(e_\p+1)\N(\p)^{-\eta e_\p}$. If $\N(\p)\ge2^{1/\eta}$, the factor is at most $(e_\p+1)2^{-e_\p}\le1$. For each of the finitely many $\p$ with $\N(\p)<2^{1/\eta}$, the factor is bounded in terms of $\eta$, since $(e+1)q^{-\eta e}$ is bounded. This gives the first bound. For the second, $\varphi(\aid)/\N(\aid)=\prod_{\p\mid\aid}(1-\N(\p)^{-1})$. There is $x_0(\eta)$ such that $1-1/x\ge x^{-\eta}$ for $x\ge x_0(\eta)$, and there are only finitely many $\p$ with $\N(\p)<x_0(\eta)$, each contributing a factor $\ge1-1/q$.
\end{proof}

\section{Hecke characters of quadratic function fields}\label{sec:hecke}

In this section we describe the Hecke characters of global function fields in the classical, ideal-theoretic language which is used for number fields (see \cite[\S3.3]{Miy} and \cite[\S2]{BDM}), and we work this description out explicitly for real quadratic function fields. The overall shape is the same as for number fields: on principal ideals, a Hecke character is a Dirichlet character times an ``infinity part'', which comes from a continuous character of the product of the completions at the infinite places, and compatibility with the units restricts the parameters of the infinity part. The difference lies in the structure of these completions: they are non-archimedean local fields, so the r\^ole of $\R^\times$ and $\mathbb{C}^\times$ is played by groups of the shape $\Z\times\mathbb{F}^\times\times(\text{pro-}p\text{ group})$. As an important consequence, the infinity types of real quadratic function fields form, up to norm twists, a finite set (Corollary~\ref{cor:finite}); this is what makes our proof of Theorem~\ref{thm:main} work. Although much of this section is standard from the idelic point of view (see \cite[Ch.~VII]{Wei74}, \cite{Tat}, \cite[Ch.~9]{Ros}, \cite{Hay}), we have not found the explicit description in the literature in a form suitable for our purposes.

\subsection{Setup}
Let $q$ be odd and let $K/k$ be a finite separable extension with full constant field $\Fq$, $\OK$ the integral closure of $A$ in $K$, and $\N(\aid)=\#(\OK/\aid)$ for non-zero integral ideals $\aid$, extended multiplicatively. Let $I$ be the group of non-zero fractional ideals of $\OK$ and $P$ the subgroup of principal ones. For an integral ideal $\f$ let
\[
K(\f)=\{\alpha\in K^\times: v_\p(\alpha)=0\text{ for all prime ideals }\p\mid\f\},
\]
let $I(\f)$ be the group of fractional ideals coprime to $\f$ and $P(\f)=\{(\alpha):\alpha\in K(\f)\}$. For $\alpha,\beta\in K(\f)$ we write $\alpha\eqs\beta\bmod\f$ if $v_\p(\alpha-\beta)\ge v_\p(\f)$ for all $\p\mid\f$. A Dirichlet character modulo $\f$ is a character $\tilde\chi$ of $(\OK/\f)^\times$; it is extended to $K(\f)$ by $\chi(\alpha):=\tilde\chi(a)$ if $a\in\OK$ and $\alpha\eqs a\bmod\f$. This is well defined and multiplicative on $K(\f)$, and we call $\chi$ a Dirichlet character modulo $\f$ as well.

Let $S_\infty=\{w_1,\dots,w_t\}$ be the set of places of $K$ above $\infty$. For $w\in S_\infty$ let $K_w$ be the completion, $d_w$ the degree of $w$ (so the residue field of $K_w$ is $\mathbb{F}_{q^{d_w}}$), $v_w$ the normalised valuation and $|x|_w=q^{-d_wv_w(x)}$ the normalised absolute value. We fix embeddings $\sigma_w:K\hookrightarrow K_w$ and write $\alpha_w=\sigma_w(\alpha)$.

\begin{lemma}[Norm and absolute values at infinity]\label{lem:prodformula}
For every $\alpha\in K^\times$ we have $\prod_{w\in S_\infty}|\alpha_w|_w=\N((\alpha))$.
\end{lemma}

\begin{proof}
For a prime ideal $\p$, the normalised absolute value is $|\alpha|_\p=\N(\p)^{-v_\p(\alpha)}$. By the product formula $\prod_v|\alpha|_v=1$ over all places of $K$, we get $\prod_{w\in S_\infty}|\alpha_w|_w=\prod_\p\N(\p)^{v_\p(\alpha)}=\N((\alpha))$.
\end{proof}

\begin{lemma}[Units]\label{lem:units}
$\OK^\times=\Fq^\times\times E$, where $E$ is a free abelian group of rank $t-1$.
\end{lemma}

\begin{proof}
This is Dirichlet's unit theorem for $S$-units with $S=S_\infty$; see \cite[Proposition~14.2]{Ros}.
\end{proof}

Thus $\Fq^\times$ plays the r\^ole of the roots of unity (for real quadratic number fields: $\pm1$) and $E$ the r\^ole of the free part generated by fundamental units.

\subsection{Continuous characters of the completions at infinity}
Fix $w\in S_\infty$ and a uniformiser $\pi_w$ of $K_w$. Every $x\in K_w^\times$ can be written uniquely as
\begin{equation}\label{eq:local-decomp}
x=\sgn_w(x)\,\pi_w^{v_w(x)}\pu{x}_w,\qquad\sgn_w(x)\in\mathbb{F}_{q^{d_w}}^\times,\ \pu{x}_w\in U_w^{(1)},
\end{equation}
where $\mathbb{F}_{q^{d_w}}^\times\subset K_w^\times$ via Teichm\"uller representatives and $U_w^{(m)}=\{x\in K_w:v_w(x-1)\ge m\}$ for $m\ge1$. The maps $\sgn_w$, $v_w$ and $\pu{\cdot}_w$ are homomorphisms, so
\[
K_w^\times\cong\Z\times\mathbb{F}_{q^{d_w}}^\times\times U_w^{(1)}.
\]
They depend on the choice of $\pi_w$, but for $c\in\Fq^\times$ we always have $\sgn_w(c)=c$, $v_w(c)=0$ and $\pu{c}_w=1$. We call $\sgn_w(x)$ the sign of $x$ (for $K_w=\Fq((1/T))$ and $\pi_w=1/T$ it is the leading coefficient of the Laurent series $x$) and $\pu{x}_w$ its principal unit part.

\begin{lemma}\label{lem:localchar}
Every continuous character $\tilde\chi_w:K_w^\times\to\T$ has the form
\begin{equation}\label{eq:localchar}
\tilde\chi_w(x)=\omega_w(\sgn_w(x))\,\eta_w(\pu{x}_w)\,|x|_w^{iv_w},
\end{equation}
where
\begin{enumerate}[label=(\arabic*)]
\item $\omega_w$ is a character of the finite cyclic group $\mathbb{F}_{q^{d_w}}^\times$;
\item $\eta_w$ is a character of $U_w^{(1)}$ which is trivial on $U_w^{(m_w)}$ for some $m_w\ge1$; it factors through the finite $p$-group $U_w^{(1)}/U_w^{(m_w)}$ of order $q^{d_w(m_w-1)}$;
\item $v_w\in\R$, determined modulo $\frac{2\pi}{d_w\log q}\Z$.
\end{enumerate}
Conversely, every such expression is a continuous character of $K_w^\times$.
\end{lemma}

\begin{proof}
By \eqref{eq:local-decomp} a character is determined by its restrictions to $\mathbb{F}_{q^{d_w}}^\times$ and $U_w^{(1)}$ and its value at $\pi_w$. The first restriction is some $\omega_w$. Continuity means that the kernel is open, so it contains a basic neighbourhood $U_w^{(m_w)}$ of $1$; this gives $\eta_w$. Finally, choose $v_w\in\R$ with $\tilde\chi_w(\pi_w)=|\pi_w|_w^{iv_w}=q^{-id_wv_w}$; the number $v_w$ is unique modulo $2\pi/(d_w\log q)$. Since $|x|_w=|\pi_w|_w^{v_w(x)}$, the formula \eqref{eq:localchar} follows. The converse is clear.
\end{proof}

\begin{remark}[Comparison with number fields]
For a real place, $\R^\times=\{\pm1\}\times\R_{>0}$ and the continuous characters are $\sgn(x)^u|x|^{iv}$ with $u\in\{0,1\}$, $v\in\R$. In \eqref{eq:localchar}, $\omega_w(\sgn_w(x))$ is the analogue of $\sgn(x)^u$ (and of $(x/|x|)^u$ at a complex place), and $|x|_w^{iv_w}$ is the analogue of $|x|^{iv}$. There are two differences:
\begin{itemize}
\item $v_w$ is only defined modulo $2\pi/(d_w\log q)$, since $|x|_w$ takes values in $q^{d_w\Z}$ only;
\item the factor $\eta_w$ has no archimedean counterpart: the connected group $\R_{>0}$ has only the characters $|x|^{iv}$, while the profinite group $U_w^{(1)}$ has infinitely many characters of finite order. Characters with $\eta_w\neq1$ are called wildly ramified at $w$.
\end{itemize}
\end{remark}

\subsection{Hecke characters on principal fractional ideals}
In analogy with the number field case, a Hecke character modulo $\f$ takes on principal fractional ideals the form
\begin{equation}\label{eq:hecke-principal}
\Hk((\alpha))=\chi(\alpha)\chi_\infty(\alpha)\qquad\text{for all }\alpha\in K(\f),
\end{equation}
where $\chi$ is a Dirichlet character modulo $\f$ (extended to $K(\f)$ as in Subsection~3.1), and the infinity part $\chi_\infty$ has the property that there is a continuous character
\[
\tilde\chi_\infty:\prod_{w\in S_\infty}K_w^\times\longrightarrow\T
\]
such that
\[
\chi_\infty(\alpha)=\tilde\chi_\infty\big((\alpha_w)_{w\in S_\infty}\big)\qquad\text{for all }\alpha\in K(\f).
\]
By Lemma~\ref{lem:localchar}, applied to each factor, the infinity part looks explicitly as follows:
\begin{equation}\label{eq:infpart}
\chi_\infty(\alpha)=\prod_{w\in S_\infty}\Big(\omega_w\big(\sgn_w(\alpha_w)\big)\,\eta_w\big(\pu{\alpha_w}_w\big)\,|\alpha_w|_w^{iv_w}\Big),
\end{equation}
where the parameters satisfy the following conditions:
\begin{enumerate}[label=(\arabic*)]
\item $\omega_w$ is a character of $\mathbb{F}_{q^{d_w}}^\times$ and $\eta_w$ is a character of $U_w^{(1)}/U_w^{(m_w)}$ for some $m_w\ge1$ (analogue of the exponents $u_j$ in the number field case);
\item $\sum_{w\in S_\infty}v_w=0$.
\end{enumerate}
As in the number field case, the second condition is not needed if we only demand that $\chi_\infty$ comes from a continuous character, but it can be forced as follows. If $\frac1t\sum_wv_w=\mu\neq0$, set $\tilde v_w=v_w-\mu$. Then $\sum_w\tilde v_w=0$ and, by Lemma~\ref{lem:prodformula},
\[
\prod_{w\in S_\infty}|\alpha_w|_w^{iv_w}=\N((\alpha))^{i\mu}\prod_{w\in S_\infty}|\alpha_w|_w^{i\tilde v_w}.
\]
Hence condition (2) can be forced at the expense of a factor $\N((\alpha))^{i\mu}$, which amounts to a vertical shift of the $L$-function by $i\mu$.

\begin{remark}
Since the $v_w$ are only determined modulo $2\pi/(d_w\log q)$, condition (2) does not remove the norm twist completely. For instance, if $t=2$ and $d_{w_1}=d_{w_2}=1$, the pairs $(v_1,v_2)$ and $(v_1+\pi/\log q,\,v_2-\pi/\log q)$ both satisfy condition (2), but the second one equals $(v_1,v_2)+(\mu,\mu)$ with $\mu=\pi/\log q$ modulo the periods, so the corresponding infinity parts differ by the norm twist $\N((\alpha))^{i\pi/\log q}=(-1)^{\deg\N((\alpha))}$. In general (all $d_w=1$), condition (2) normalises the norm twist only up to the finite group of twists $\N(\cdot)^{2\pi ij/(t\log q)}$, $j\in\Z$. This phenomenon has no counterpart for number fields; see Corollary~\ref{cor:finite} for its effect in the real quadratic case.
\end{remark}

If $\chi$ is given, then $\chi_\infty$ needs to be defined in such a way that
\begin{equation}\label{eq:unitcond}
\chi(u)\chi_\infty(u)=1\qquad\text{for all units }u\in\OK^\times,
\end{equation}
so that $\Hk((\alpha))=\chi(\alpha)\chi_\infty(\alpha)$ is well defined on the principal ideals $(\alpha)$. By Lemma~\ref{lem:units}, \eqref{eq:unitcond} has to be checked for $u\in\Fq^\times$ and for a basis of $E$. For constants $c\in\Fq^\times$ we have $\sgn_w(c)=c$, $\pu{c}_w=1$, $|c|_w=1$ for all $w$, so \eqref{eq:unitcond} for constants reads
\begin{equation}\label{eq:constcond}
\chi(c)\prod_{w\in S_\infty}\omega_w(c)=1\qquad\text{for all }c\in\Fq^\times,
\end{equation}
and it suffices to check this for a generator $c_0$ of the cyclic group $\Fq^\times$. This is the analogue of the condition $\chi\chi_\infty(-1)=1$, which fixed the parity of $u_1+u_2$ in the real quadratic number field case. We again say that $\Hk$ belongs to $\chi$.

The following proposition shows that this definition captures exactly the idelic Hecke characters, i.e.\ continuous characters of $\mathbb{A}_K^\times/K^\times$ (see \cite[Ch.~VII]{Wei74}); it is the reason why the infinity part must come from a continuous character of $\prod_wK_w^\times$.

\begin{proposition}\label{prop:idelic}
Let $\xi:\mathbb{A}_K^\times/K^\times\to\T$ be continuous, with components $\xi_v$, and let $\f$ be its finite conductor. Put $\Hk(\p)=\xi_\p(\varpi_\p)$ for $\p\nmid\f$, $\chi(\alpha)=\prod_{\p\mid\f}\xi_\p(\alpha)^{-1}$ and $\chi_\infty(\alpha)=\prod_{w\in S_\infty}\xi_w(\alpha_w)^{-1}$ for $\alpha\in K(\f)$. Then \eqref{eq:hecke-principal} holds. Conversely, every character $\Hk$ of $I(\f)$ satisfying \eqref{eq:hecke-principal} with $\chi_\infty$ as in \eqref{eq:infpart} arises in this way from a unique idelic Hecke character.
\end{proposition}

\begin{proof}
Since $\xi_\p$ is trivial on $\mathcal{O}_\p^\times$ for $\p\nmid\f$, we have $\xi_\p(\alpha)=\Hk(\p)^{v_\p(\alpha)}$. For $\p\mid\f$ and $\alpha\in K(\f)$, $\xi_\p(\alpha)$ depends only on $\alpha$ modulo $1+\p^{v_\p(\f)}$, so $\chi$ is a Dirichlet character modulo $\f$. As $\xi$ is trivial on $K^\times$, $1=\prod_v\xi_v(\alpha)=\Hk((\alpha))\chi(\alpha)^{-1}\chi_\infty(\alpha)^{-1}$, which is \eqref{eq:hecke-principal}.

Conversely, let $x\in\mathbb{A}_K^\times$. By the approximation theorem choose $\beta\in K^\times$ with $\beta x_\p\in1+\p^{v_\p(\f)}\mathcal{O}_\p$ for all $\p\mid\f$, put $y=\beta x$ and define $\xi(x)=\Hk\big(\prod_{\p\nmid\f}\p^{v_\p(y_\p)}\big)\prod_w\tilde\chi_w(y_w)^{-1}$, where $\tilde\chi_w$ are the local factors in \eqref{eq:infpart}. If $\beta'$ is another choice, then $\delta=\beta'/\beta\in K(\f)$ satisfies $\delta\eqs1\bmod\f$, so $\chi(\delta)=1$ and the two definitions differ by $\Hk((\delta))\chi_\infty(\delta)^{-1}=\chi(\delta)=1$. Thus $\xi(x)$ is well defined; it is a homomorphism, trivial on an open subgroup, and trivial on $K^\times$ (take $\beta=x^{-1}$ for $x\in K^\times$).
\end{proof}

\subsection{The case of real quadratic extensions}
Now let $K=k(\sqrt D)$ be real quadratic, with $D$ monic of degree $2g+2$, embeddings $\sigma_1,\sigma_2$, fundamental unit $\epsilon$ normalised by \eqref{eq:eps-norm} and regulator $R$ as in Subsection~\ref{subsec:real}. Here $K_{\infty_1}=K_{\infty_2}=\kinf$, and we use the uniformiser $1/T$ at both places, so that $\sgn$, $\deg$ and $\pu{\cdot}$ are as in \eqref{eq:decomp} and $|x|=q^{\deg x}$. In the number field case one normalises $\epsilon>1$; the analogue is \eqref{eq:eps-norm}, and $\log|\sigma_1(\epsilon)|=R\log q$ plays the r\^ole of $\log\epsilon$. The second identity in \eqref{eq:eps-sigma2} is the analogue of $\sgn(\sigma_2(\epsilon))=N(\epsilon)=\pm1$ for real quadratic number fields.

\emph{The shape of $\chi_\infty$.} Here $t=2$ and $d_{\infty_1}=d_{\infty_2}=1$, so condition (2) means $v_{\infty_1}=-v_{\infty_2}=:v$, and \eqref{eq:infpart} takes the form
\begin{equation}\label{eq:infpart-real}
\chi_\infty(\alpha)=\Big|\frac{\sigma_1(\alpha)}{\sigma_2(\alpha)}\Big|^{iv}\omega_1\big(\sgn(\sigma_1(\alpha))\big)\,\omega_2\big(\sgn(\sigma_2(\alpha))\big)\,\eta_1\big(\pu{\sigma_1(\alpha)}\big)\,\eta_2\big(\pu{\sigma_2(\alpha)}\big)
\end{equation}
for a real number $v$, characters $\omega_1,\omega_2$ of $\Fq^\times$ (replacing $u_1,u_2\in\{0,1\}$) and characters $\eta_1,\eta_2$ of $U^{(1)}$ of finite level. Note that
\[
\Big|\frac{\sigma_1(\alpha)}{\sigma_2(\alpha)}\Big|^{iv}=\exp\big(iv\log q\,(\deg\sigma_1(\alpha)-\deg\sigma_2(\alpha))\big).
\]
The unit group $\OK^\times$ is generated by $\epsilon$ and a generator $c_0$ of $\Fq^\times$, and therefore it suffices to ensure that $\chi\chi_\infty(\epsilon)=1$ and $\chi\chi_\infty(c_0)=1$.

\emph{The condition at $\epsilon$.} Suppose that
\begin{equation}\label{eq:gamma}
\chi(\epsilon)\,\omega_2\big(\Norm(\epsilon)\big)\,\frac{\eta_1(\pu{\sigma_1(\epsilon)})}{\eta_2(\pu{\sigma_1(\epsilon)})}=\exp(-2\pi i\gamma),
\end{equation}
i.e.
\[
\gamma=-\frac{1}{2\pi}\arg\Big(\chi(\epsilon)\,\omega_2\big(\Norm(\epsilon)\big)\frac{\eta_1(\pu{\sigma_1(\epsilon)})}{\eta_2(\pu{\sigma_1(\epsilon)})}\Big)\in[0,1).
\]
The left-hand side is a product of values of characters of finite order, so $\gamma\in\mathbb{Q}$. By \eqref{eq:eps-norm} we have $\omega_1(\sgn(\sigma_1(\epsilon)))=\omega_1(1)=1$, and by \eqref{eq:eps-sigma2} the remaining sign and principal-unit factors of $\chi_\infty(\epsilon)$ are exactly those in \eqref{eq:gamma}. To ensure that $\chi\chi_\infty(\epsilon)=1$, it therefore suffices to demand that
\[
\Big|\frac{\sigma_1(\epsilon)}{\sigma_2(\epsilon)}\Big|^{iv}=\exp(2\pi i\gamma),
\]
i.e., since $|\sigma_1(\epsilon)/\sigma_2(\epsilon)|=q^{2R}=|\sigma_1(\epsilon)|^2$,
\[
|\sigma_1(\epsilon)|^{2iv}=\exp(2ivR\log q)=\exp(2\pi i\gamma).
\]
This is equivalent to
\begin{equation}\label{eq:v}
v=\frac{\pi(n+\gamma)}{R\log q}=\frac{\pi(n+\gamma)}{\log|\sigma_1(\epsilon)|}\qquad\text{for some }n\in\Z.
\end{equation}

\emph{The condition at constants.} By \eqref{eq:constcond}, we need
\[
\chi(c)\,\omega_1(c)\,\omega_2(c)=1\qquad\text{for all }c\in\Fq^\times,
\]
i.e.\ $\omega_2=\overline{\chi|_{\Fq^\times}\omega_1}$. Thus $\omega_1$ can be chosen freely among the $q-1$ characters of $\Fq^\times$, and then $\omega_2$ is determined. This is the analogue of the rule ``$u_1=u_2$ if $\chi(-1)=1$ and $u_1\neq u_2$ if $\chi(-1)=-1$'': there, $u_1\in\{0,1\}$ is free and $u_2$ is determined by the parity of $\chi$.

So, altogether, the infinity part is of the form
\begin{multline}\label{eq:infpart-full}
\chi_\infty(\alpha)=\Big|\frac{\sigma_1(\alpha)}{\sigma_2(\alpha)}\Big|^{i\pi(n+\gamma)/\log|\sigma_1(\epsilon)|}\cdot\omega_1\big(\sgn(\sigma_1(\alpha))\big)\\
\cdot\omega_2\big(\sgn(\sigma_2(\alpha))\big)\cdot\eta_1\big(\pu{\sigma_1(\alpha)}\big)\cdot\eta_2\big(\pu{\sigma_2(\alpha)}\big)
\end{multline}
with $\omega_2=\overline{\chi|_{\Fq^\times}\omega_1}$ and $\gamma$ as in \eqref{eq:gamma}. Equivalently,
\begin{equation}\label{eq:size-factor}
\Big|\frac{\sigma_1(\alpha)}{\sigma_2(\alpha)}\Big|^{i\pi(n+\gamma)/\log|\sigma_1(\epsilon)|}=\exp\Big(\frac{i\pi(n+\gamma)}{R}\big(\deg\sigma_1(\alpha)-\deg\sigma_2(\alpha)\big)\Big).
\end{equation}
Formula \eqref{eq:infpart-full} is the precise analogue of the corresponding formula for real quadratic number fields (see \cite[\S2]{BDM}), with $\log\epsilon$ replaced by $\log|\sigma_1(\epsilon)|$. We summarise:

\begin{theorem}\label{thm:hecke-real}
Let $K$ be a real quadratic function field as above and $\chi$ a Dirichlet character modulo $\f$. The Hecke characters on $P(\f)$ belonging to $\chi$ are, up to a factor $\N(\cdot)^{i\mu}$, exactly the maps $(\alpha)\mapsto\chi(\alpha)\chi_\infty(\alpha)$ with $\chi_\infty$ as in \eqref{eq:infpart-full}, where $\omega_1$ is any character of $\Fq^\times$, $\eta_1,\eta_2$ are any characters of $U^{(1)}$ of finite level, and $n\in\Z$.
\end{theorem}

\begin{proof}
The derivation above shows that \eqref{eq:unitcond} holds for $\epsilon$ and $c_0$, hence for all units, if and only if the parameters are as stated; the reduction to $\sum_wv_w=0$ costs a factor $\N(\cdot)^{i\mu}$. The resulting map is multiplicative on $K(\f)$ and trivial on units, hence a character of $K(\f)/\OK^\times\cong P(\f)$.
\end{proof}

\begin{corollary}[Finiteness of the parameter $n$]\label{cor:finite}
Fix $\chi,\omega_1,\eta_1,\eta_2$. Then \eqref{eq:infpart-full} yields exactly $2R$ distinct characters of $P(\f)$, indexed by $n\bmod2R$, and the characters for $n$ and $n+R$ differ by the factor $\N((\alpha))^{i\pi/\log q}=(-1)^{\deg\N((\alpha))}$. In particular, modulo norm twists there are exactly $R$ choices of $n$, and every Hecke character of $K$ is a character of finite order times a norm twist $\N(\cdot)^{i\mu}$.
\end{corollary}

\begin{proof}
By \eqref{eq:size-factor}, replacing $n$ by $n'$ multiplies $\chi_\infty$ by $\exp(i\pi(n'-n)(\deg\sigma_1\alpha-\deg\sigma_2\alpha)/R)$. By the approximation theorem there is $\alpha\in K^\times$ with $v_\p(\alpha)=0$ for $\p\mid\f$, $v_{\infty_1}(\alpha)=0$ and $v_{\infty_2}(\alpha)=-1$, so $\deg\sigma_1\alpha-\deg\sigma_2\alpha$ takes the value $-1$, hence every integer value. Therefore the factor is trivial if and only if $2R\mid n'-n$. For $n'=n+R$ the factor is $(-1)^{\deg\sigma_1\alpha-\deg\sigma_2\alpha}=(-1)^{\deg\sigma_1\alpha+\deg\sigma_2\alpha}=(-1)^{\deg\N((\alpha))}$ by Lemma~\ref{lem:prodformula}. All other factors in \eqref{eq:infpart-full} are of finite order, and so is the root of unity \eqref{eq:size-factor}.

For the last statement, let $\Hk'$ be a Hecke character modulo $\f$ on $I(\f)$. By Theorem~\ref{thm:hecke-real} and what we have just shown, its restriction to $P(\f)$ equals $\N(\cdot)^{i\mu}\phi$ with $\phi$ of finite order, say $\phi^m=1$. Then $\Hk'':=\Hk'\N(\cdot)^{-i\mu}$ is a character of $I(\f)$ with $(\Hk'')^m$ trivial on $P(\f)$, so $(\Hk'')^m$ is a character of the finite group $I(\f)/P(\f)$ (Lemma~\ref{lem:raycl} below), and hence $\Hk''$ has finite order.
\end{proof}

\begin{example}[Characters of sign type]
If $\eta_1=\eta_2=1$, then $\gamma$ is given by $\chi(\epsilon)\omega_2(\Norm(\epsilon))=\exp(-2\pi i\gamma)$ and
\[
\chi_\infty(\alpha)=\Big|\frac{\sigma_1(\alpha)}{\sigma_2(\alpha)}\Big|^{i\pi(n+\gamma)/\log|\sigma_1(\epsilon)|}\omega_1\big(\sgn(\sigma_1(\alpha))\big)\,\omega_2\big(\sgn(\sigma_2(\alpha))\big),
\]
which is formally identical to the number field formula. If moreover $\f=(1)$ and $\omega_1=\omega_2=1$, then $\gamma=0$ and
\begin{equation}\label{eq:psi-n}
\chi_\infty(\alpha)=\Big|\frac{\sigma_1(\alpha)}{\sigma_2(\alpha)}\Big|^{i\pi n/\log|\sigma_1(\epsilon)|}=\exp\Big(\frac{i\pi n}{R}\big(\deg\sigma_1(\alpha)-\deg\sigma_2(\alpha)\big)\Big),
\end{equation}
a family of Hecke characters modulo $(1)$. For $n=1$ we obtain the character
\begin{equation}\label{eq:psi}
\psi((\alpha)):=\exp\Big(\frac{i\pi}{R}\big(\deg\sigma_1(\alpha)-\deg\sigma_2(\alpha)\big)\Big),
\end{equation}
which is well defined on principal ideals since multiplication by $\epsilon$ changes $\deg\sigma_1-\deg\sigma_2$ by $2R$; its powers $\psi^n$ are the characters \eqref{eq:psi-n}. It is unramified at $\infty_1$ and $\infty_2$ and has order $2R$.
\end{example}

\subsection{Extension to general fractional ideals}
Let $\f$ be an integral ideal of $\OK$, and let $I$, $P$, $I(\f)$, $P(\f)$ be as in Subsection~3.1.

\begin{lemma}\label{lem:raycl}
The natural map $I(\f)/P(\f)\to I/P$ is an isomorphism, and $I/P=\Cl(\OK)$ is finite. We write $h_K=\#\Cl(\OK)$ for the ideal class number.
\end{lemma}

\begin{proof}
Injectivity: $P(\f)=P\cap I(\f)$, since a principal ideal $(\alpha)$ coprime to $\f$ has $\alpha\in K(\f)$. Surjectivity: given $\aid\in I$, choose by the approximation theorem $\beta\in K^\times$ with $v_\p(\beta)=-v_\p(\aid)$ for all $\p\mid\f$; then $\beta\aid\in I(\f)$. Finiteness follows from the finiteness of the divisor class group $\Pic^0(K)$ and Remark~\ref{rem:twoclass} below.
\end{proof}

\begin{remark}[Two class numbers]\label{rem:twoclass}
For function fields one has to distinguish $h_K$ from the divisor class number $\tilde h_K=\#\Pic^0(K)$, which counts degree-zero divisors on the whole curve, including the infinite places. For quadratic $K$:
\[
\tilde h_K=R\,h_K\ \text{(real)},\qquad h_K=2\tilde h_K\ \text{(inert)},\qquad h_K=\tilde h_K\ \text{(ramified)}.
\]
Indeed, forgetting the infinite part of a divisor gives a map $\Pic^0(K)\to\Cl(\OK)$. Its kernel consists of the classes of degree-zero divisors supported at infinity. In the real case these are the multiples of $\infty_1-\infty_2$, and $n(\infty_1-\infty_2)=\mathrm{div}(\alpha)$ forces $\alpha=c\epsilon^j$, while $\mathrm{div}(\epsilon)=-R\infty_1+R\infty_2$; so the kernel is cyclic of order $R$. In the imaginary cases the kernel is trivial. The cokernel is $\Z/\gcd_w(d_w)\Z$, which is trivial except in the inert case, where it has order $2$. The class number appearing in the extension process below is $h_K$, not $\tilde h_K$.
\end{remark}

In general, a Hecke character modulo $\f$ is a function on $I$ which equals zero on $I\setminus I(\f)$, whose restriction to $I(\f)$ is a character of $I(\f)$, and whose restriction to $P(\f)$ is a character as described in Subsections~3.3--3.4. The extension from $P(\f)$ to $I(\f)$ works exactly as for number fields (see \cite[\S2]{BDM}), with the class number replaced by $h_K$. Write
\[
I(\f)/P(\f)\cong\Z_{m_1}\oplus\dots\oplus\Z_{m_k},
\]
and choose ideals $\aid_1,\dots,\aid_k\in I(\f)$ whose classes correspond to the standard generators. Then $\aid_l^{m_l}\in P(\f)$, every $\mathfrak{b}\in I(\f)$ can be written uniquely as
\[
\mathfrak{b}=\aid_1^{j_1}\cdots\aid_k^{j_k}(\alpha),\qquad0\le j_l\le m_l-1,\ \alpha\in K(\f)
\]
(the ideal $(\alpha)$ being unique, $\alpha$ unique up to units), and for a fixed choice of $m_l$-th roots and any character $\lambda$ of $I(\f)/P(\f)$,
\begin{equation}\label{eq:extension}
\Hk(\mathfrak{b})=\lambda\big(\aid_1^{j_1}\cdots\aid_k^{j_k}P(\f)\big)\,\Hk\big(\aid_1^{m_1}\big)^{j_1/m_1}\cdots\Hk\big(\aid_k^{m_k}\big)^{j_k/m_k}\,\Hk((\alpha))
\end{equation}
defines a Hecke character on $I(\f)$.

\begin{proposition}\label{prop:extension}
Every Hecke character on $I(\f)$ whose restriction to $P(\f)$ is a given $\Hk$ is of the form \eqref{eq:extension}. Hence every Hecke character on $P(\f)$ has exactly $h_K$ extensions to $I(\f)$, one for each class group character $\lambda$.
\end{proposition}

\begin{proof}
One checks directly that \eqref{eq:extension} is multiplicative: multiplying two representations, the exponents $j_l$ add, and whenever $j_l+j_l'\ge m_l$ one factor $\aid_l^{m_l}$ moves into the principal part, where it contributes $\Hk(\aid_l^{m_l})$, matching $(\Hk(\aid_l^{m_l})^{1/m_l})^{m_l}$. If $\chi'$ is any extension, then $\chi'/\Hk$ (with $\Hk$ given by \eqref{eq:extension} with $\lambda=1$) is trivial on $P(\f)$, hence equals some class group character $\lambda$. Distinct $\lambda$ give distinct extensions, and there are $\#(I(\f)/P(\f))=h_K$ of them by Lemma~\ref{lem:raycl}.
\end{proof}

As before, a change of the $m_l$-th roots amounts to a change of $\lambda$, class group characters are Hecke characters modulo $(1)$, and the principal character $\chi_{\mathrm{Hecke},0}$ (equal to $1$ on $I(\f)$ and $0$ elsewhere) belongs to the principal Dirichlet character $\chi_0$ modulo $\f$.

\subsection{Values at the infinite places and $L$-functions}
A feature without direct counterpart in the ideal-theoretic language is that the infinite places are places of the same kind as the finite ones, so the Hecke character also has a value there, which enters the Euler product.

\begin{definition}\label{def:infvalue}
For $w\in S_\infty$ we say that $\Hk$ is unramified at $w$ if $\omega_w$ and $\eta_w$ in \eqref{eq:infpart} are trivial. In this case we put
\[
\Hk(w):=|\pi_w|_w^{-iv_w}=q^{id_wv_w};
\]
otherwise we put $\Hk(w):=0$. The conductor exponent $c_w$ is $0$ if $\Hk$ is unramified at $w$, $1$ if $\eta_w=1\neq\omega_w$, and $\min\{m\ge2:\eta_w|_{U_w^{(m)}}=1\}$ if $\eta_w\neq1$.
\end{definition}

This normalisation is forced: if $\Hk$ is unramified everywhere (so $\f=(1)$, $\chi=1$, all $\omega_w,\eta_w$ trivial), then it extends to all divisors of $K$ and must be trivial on principal divisors $\mathrm{div}(\alpha)=(\alpha)+\sum_wv_w(\alpha)\,w$, and indeed
\[
\Hk((\alpha))\prod_w\Hk(w)^{v_w(\alpha)}=\prod_w|\alpha_w|_w^{iv_w}q^{id_wv_wv_w(\alpha)}=\prod_wq^{-id_wv_wv_w(\alpha)}q^{id_wv_wv_w(\alpha)}=1.
\]
(Here the first $v_w$ is the parameter, the second the valuation.)

The finite and complete Hecke $L$-functions are
\[
L(s,\Hk)=\sum_{\substack{\aid\subseteq\OK\\\aid\neq0}}\frac{\Hk(\aid)}{\N(\aid)^s}=\prod_\p\big(1-\Hk(\p)\N(\p)^{-s}\big)^{-1},
\]
\[
\Lambda(s,\Hk)=L(s,\Hk)\prod_{w\in S_\infty}\big(1-\Hk(w)q^{-d_ws}\big)^{-1}.
\]

\begin{theorem}[Functional equation]\label{thm:FE}
Let $\Hk$ be primitive, i.e.\ induced by its own idelic character (Proposition~\ref{prop:idelic}) with $\f$ the finite conductor, let $\mathfrak{m}=\f+\sum_wc_ww$ be its conductor, with degree $\deg\mathfrak{m}=\log_q\N(\f)+\sum_wc_wd_w$, and let $g$ be the genus of $K$. If $\Hk$ is not of the form $\N(\cdot)^{s_0}$ (extended to all divisors), then $\Lambda(s,\Hk)$ is a polynomial in $q^{-s}$ of degree $N=2g-2+\deg\mathfrak{m}$, and
\[
\Lambda(s,\Hk)=W(\Hk)\,q^{N(\frac12-s)}\,\Lambda(1-s,\overline{\Hk}),\qquad|W(\Hk)|=1.
\]
\end{theorem}

\begin{proof}
See \cite[Ch.~VII]{Wei74} or \cite{Tat}; by Proposition~\ref{prop:idelic} our $\Hk$ is an idelic Hecke character, and our Euler factors at infinity are the local $L$-factors there.
\end{proof}

In the real quadratic case, $g=\frac{\deg D}{2}-1$ and $N=\deg D-4+\log_q\N(\f)+c_1+c_2$. If $\Hk$ is unramified at $\infty_1$ and $\infty_2$ (i.e.\ $\omega_1=\omega_2=1$, which forces $\chi|_{\Fq^\times}=1$, and $\eta_1=\eta_2=1$), then by \eqref{eq:v}
\[
\Hk(\infty_1)=q^{iv}=\exp\Big(\frac{i\pi(n+\gamma)}{R}\Big),\qquad\Hk(\infty_2)=q^{-iv}=\exp\Big(-\frac{i\pi(n+\gamma)}{R}\Big),
\]
and the Euler factors at infinity are
\[
\big(1-e^{i\pi(n+\gamma)/R}q^{-s}\big)^{-1}\big(1-e^{-i\pi(n+\gamma)/R}q^{-s}\big)^{-1}.
\]

\subsection{Consistency with the previous definition in the imaginary quadratic case}\label{subsec:consistency}
The imaginary quadratic case was treated by Baier and the author in \cite[\S4]{BaMo}, where Hecke characters were introduced in an ad hoc way, adapted to that situation. In this subsection we show that this previous definition is consistent with the general description of Hecke characters given above: the characters used in \cite{BaMo} are exactly the Hecke characters of trivial infinity type.

Let $K=k(\sqrt D)$ be imaginary quadratic, i.e.\ $\deg D\ge1$ is odd, or $\deg D$ is even and $\sgn(D)$ is a non-square in $\Fq^\times$. Then there is exactly one place $\infty_K$ of $K$ above $\infty$, of degree $2$ in the inert case and of degree $1$ in the ramified case, and $\OK^\times=\Fq^\times$ by Lemma~\ref{lem:units} with $t=1$. Let $\mathcal{P}$ denote the set of non-zero prime ideals of $\OK$, $\f$ an integral ideal of $\OK$, and $U(\f)=\{u+\f:u\in\OK^\times\}\subseteq(\OK/\f)^\times$.

\emph{The previous definition.} In \cite[\S\S4.5--4.6]{BaMo}, Hecke characters were set up as follows.
\begin{enumerate}[label=(\roman*)]
\item Let $G(\f)$ be the group of Dirichlet characters $\chi$ modulo $\f$ which are trivial on the units. For $\chi\in G(\f)$, the function $\tilde\chi((c)):=\chi(c)$ is well defined on the principal ideals coprime to $\f$, since $\chi$ is constant on the generators of a principal ideal. Thus $G(\f)$ is identified with the group of characters of $(\OK/\f)^\times/U(\f)$.
\item Each $\tilde\chi$ extends in exactly $h_K$ ways to a character of the group of all fractional ideals coprime to $\f$, the extensions differing by class group characters. These extensions form a group $H(\f)$ of order $h_K\cdot\#G(\f)$.
\item The characters $\chi\in H(\f)$ are turned into Hecke characters on divisors (see \cite[Chapter~9]{Ros}) as follows. For $\p\in\mathcal{P}$ let $P$ be the maximal ideal of the discrete valuation ring $(\OK)_\p$; these $P$ are precisely the finite primes of $K$, and $\infty_K$ is the only other prime. Let $\mathfrak{F}=\sum_{i=1}^{\omega}e_iP_i$, where $\f=\prod_{i=1}^\omega\p_i^{e_i}$ is the prime factorisation of $\f$ and $P_i$ corresponds to $\p_i$. Put
\[
\lambda(P):=\chi(\p)\ \text{for the finite prime $P$ corresponding to }\p,\qquad\lambda(\infty_K):=1.
\]
Then $\lambda$ extends uniquely to a Hecke character modulo $\mathfrak{F}$, with $L$-function
\begin{equation}\label{eq:old-L}
L(s,\lambda)=\prod_P\big(1-\lambda(P)\N(P)^{-s}\big)^{-1}=\prod_{\p\in\mathcal{P}}\big(1-\chi(\p)\N(\p)^{-s}\big)^{-1}\big(1-\N(\infty_K)^{-s}\big)^{-1}\quad(\Re s>1),
\end{equation}
where $\N(\infty_K)=q^{\deg\infty_K}$, i.e.\ $q^2$ in the inert case and $q$ in the ramified case.
\end{enumerate}
The analytic input in \cite{BaMo} was then the following result \cite[Theorem~9.24A]{Ros}, together with the Riemann hypothesis for $L(s,\lambda)$.

\begin{theorem}\label{thm:rosen}
Let $\lambda$ be a primitive Hecke character with conductor $\mathfrak{F}$, and suppose that $\lambda$ is not trivial on $D_0(S)$, the group of divisors of degree $0$ with support disjoint from the set $S$ of primes on which $\mathfrak{F}$ is supported. Then $L(s,\lambda)$ is a polynomial in $q^{-s}$ of degree $2g-2+\deg_K\mathfrak{F}$, where $g$ is the genus of $K$. Moreover, $\Lambda(s,\lambda):=q^{(g-1)s}\N(\mathfrak{F})^{s/2}L(s,\lambda)$ satisfies $\Lambda(s,\lambda)=\varepsilon(\lambda)\Lambda(1-s,\overline\lambda)$ with $|\varepsilon(\lambda)|=1$, and all zeros of $L(s,\lambda)$ have real part $1/2$.
\end{theorem}

\emph{The general description in the imaginary quadratic case.} Since $t=1$, the size factor $|\alpha|_{\infty_K}^{iv}=\N((\alpha))^{iv}$ is a norm twist by Lemma~\ref{lem:prodformula}, so after the normalisation of Subsection~3.3 the infinity part \eqref{eq:infpart} takes the form
\begin{equation}\label{eq:infpart-imag}
\chi_\infty(\alpha)=\omega\big(\sgn_{\infty_K}(\alpha)\big)\,\eta\big(\pu{\alpha}_{\infty_K}\big),
\end{equation}
where $\omega$ is a character of $\mathbb{F}_{q^{\deg\infty_K}}^\times$ and $\eta$ a character of $U^{(1)}_{\infty_K}$ of finite level. This is the analogue of $\chi_\infty(\alpha)=(\alpha/|\alpha|)^u$ for imaginary quadratic number fields. Since the unit group is $\Fq^\times$, the unit condition \eqref{eq:constcond} reduces to
\begin{equation}\label{eq:imag-cond}
\omega(c)=\overline{\chi(c)}\qquad(c\in\Fq^\times),
\end{equation}
and there is no condition corresponding to a fundamental unit. We say that $\Hk$ has trivial infinity type if $\omega=1$ and $\eta=1$, i.e.\ $\chi_\infty\equiv1$ and $\Hk((\alpha))=\chi(\alpha)$ on $P(\f)$.

\begin{proposition}[Consistency]\label{prop:consistency}
Let $K$ be imaginary quadratic. Then the following hold.
\begin{enumerate}[label=(\arabic*)]
\item A Hecke character on $P(\f)$ belonging to $\chi$ has trivial infinity type if and only if $\chi$ is trivial on $\OK^\times$. Hence the Hecke characters of trivial infinity type on $P(\f)$ are exactly the characters $\tilde\chi$, $\chi\in G(\f)$, of (i).
\item The extensions of these characters to $I(\f)$ given by Proposition~\ref{prop:extension} are exactly the elements of $H(\f)$ in (ii). Thus $H(\f)$ is the set of Hecke characters modulo $\f$ of trivial infinity type, and $\#H(\f)=h_K\cdot\#G(\f)$.
\item Every $\chi\in H(\f)$ is unramified at $\infty_K$, and its value there in the sense of Definition~\ref{def:infvalue} is $\Hk(\infty_K)=1$. Hence the choice $\lambda(\infty_K)=1$ in (iii) is not a convention but forced, and the divisor character $\lambda$ of (iii) is the Hecke character $\Hk$ extended to divisors.
\item Consequently, $L(s,\lambda)$ in \eqref{eq:old-L} is the complete $L$-function $\Lambda(s,\Hk)$ of Subsection~3.6, and Theorem~\ref{thm:rosen} is the special case of Theorem~\ref{thm:FE} for Hecke characters of trivial infinity type of an imaginary quadratic function field.
\end{enumerate}
\end{proposition}

\begin{proof}
(1) By \eqref{eq:imag-cond}, $\omega=1$ is possible if and only if $\chi$ is trivial on $\Fq^\times=\OK^\times$, and then $\eta=1$ may be chosen. In this case $\Hk((\alpha))=\chi(\alpha)=\tilde\chi((\alpha))$.

(2) Class group characters are Hecke characters modulo $(1)$ of trivial infinity type, so multiplying by them preserves trivial infinity type. The claim then follows from Proposition~\ref{prop:extension} and Lemma~\ref{lem:raycl}.

(3) Since $\omega=\eta=1$ and $v=0$, $\Hk$ is unramified at $\infty_K$ and $\Hk(\infty_K)=q^{i\deg(\infty_K)\cdot0}=1$ by Definition~\ref{def:infvalue}. Independently of this definition, the value is forced: if $\alpha\eqs1\bmod\f$, then $\Hk((\alpha))=\chi(\alpha)=1$, so triviality of the Hecke character on the principal divisor $\mathrm{div}(\alpha)=(\alpha)+v_{\infty_K}(\alpha)\infty_K$ gives $\lambda(\infty_K)^{v_{\infty_K}(\alpha)}=1$. By the approximation theorem there is such an $\alpha$ with $v_{\infty_K}(\alpha)=1$, whence $\lambda(\infty_K)=1$.

(4) By Proposition~\ref{prop:idelic}, $\Hk$ is a genuine (idelic) Hecke character. By (3), the Euler factor of $\Lambda(s,\Hk)$ at $\infty_K$ is $(1-q^{-\deg(\infty_K)s})^{-1}=(1-\N(\infty_K)^{-s})^{-1}$, which is the last factor in \eqref{eq:old-L}.
\end{proof}

\subsection{A prime number theorem for Hecke characters}
We say that a Hecke character $\Hk$ modulo $\f$ is a norm twist if there is $s_0\in\mathbb{C}$ such that $\Hk(\aid)=\N(\aid)^{s_0}$ for all $\aid\in I(\f)$. If $\Hk$ is not a norm twist, then the primitive idelic character inducing it (Proposition~\ref{prop:idelic}) is non-trivial on the ideles of degree zero, so Theorem~\ref{thm:FE} applies to it. Moreover, by Weil's proof of the Riemann hypothesis for curves over finite fields \cite{Wei48} (see also \cite[Ch.~9]{Ros}), all inverse zeros of the polynomial $\Lambda(s,\Hk)$ in $q^{-s}$ have absolute value $q^{1/2}$ if $\Hk$ is unitary and primitive. The following lemma is a standard consequence; we include the proof in order to make the dependence on the modulus explicit.

\begin{lemma}\label{lem:PNT}
Let $K$ be a global function field with constant field $\Fq$ and genus $g$, and let $t=\#S_\infty$. Let $\Hk$ be a unitary Hecke character modulo $\f$ whose conductor has degree at most $\log_q\N(\f)+c_\infty$, where $c_\infty=\sum_{w\in S_\infty}c_wd_w$. Then for every $N\ge1$ the following hold.
\begin{enumerate}[label=(\roman*)]
\item If $\Hk$ is not a norm twist, then
\[
\Big|\sum_{\substack{\p\nmid\f\\\N(\p)=q^N}}\Hk(\p)\Big|\le\big(10g+3+\log_q\N(\f)+c_\infty\big)\frac{q^{N/2}}{N}+t+\omega(\f).
\]
\item We have
\[
\Big|\#\{\p\nmid\f:\N(\p)=q^N\}-\frac{q^N}{N}\Big|\le(10g+4)\frac{q^{N/2}}{N}+t+\omega(\f).
\]
\end{enumerate}
\end{lemma}

\begin{proof}
Let $B_d$ be the number of places of $K$ of degree $d$. By the Riemann hypothesis for $K$, we have $\sum_{e\mid d}eB_e=q^d+1-\sum_{j=1}^{2g}\omega_j^d$ with $|\omega_j|=q^{1/2}$, hence $dB_d\le q^d+1+2gq^{d/2}$ and
\begin{equation}\label{eq:smallplaces}
\sum_{\substack{d\mid N\\d<N}}dB_d\le\sum_{d\le N/2}\big(q^d+1+2gq^{d/2}\big)\le(3+8g)q^{N/2},
\end{equation}
using $q\ge3$. Taking $d=N$ in the first identity gives $|NB_N-q^N|\le1+2gq^{N/2}+(3+8g)q^{N/2}$, and (ii) follows after removing the at most $t$ infinite places and the at most $\omega(\f)$ prime ideals dividing $\f$.

For (i), let $\xi$ be the primitive idelic character inducing $\Hk$ and $\mathfrak{m}$ its conductor, so that $\deg\mathfrak{m}\le\log_q\N(\f)+c_\infty$ and $\xi(\p)=\Hk(\p)$ for $\p\nmid\f$. By Theorem~\ref{thm:FE} and the Riemann hypothesis, with $u=q^{-s}$,
\[
\prod_{v\notin\supp\mathfrak{m}}\big(1-\xi(v)u^{\deg v}\big)^{-1}=\prod_{i=1}^{2g-2+\deg\mathfrak{m}}(1-\alpha_iu),\qquad|\alpha_i|=q^{1/2},
\]
the product on the left being over all places of $K$ (finite and infinite) outside the support of $\mathfrak{m}$. Comparing the coefficients of $u^N$ in the logarithmic derivatives, we get
\[
\sum_{\substack{v\notin\supp\mathfrak{m}\\\deg v\mid N}}\deg v\cdot\xi(v)^{N/\deg v}=-\sum_i\alpha_i^N.
\]
The terms with $\deg v<N$ contribute at most $(3+8g)q^{N/2}$ in absolute value by \eqref{eq:smallplaces}. Hence the sum of $\xi(v)$ over the places $v\notin\supp\mathfrak{m}$ of degree $N$ is at most $(2g-2+\deg\mathfrak{m}+3+8g)q^{N/2}/N$ in absolute value. Removing the at most $t$ infinite places and the at most $\omega(\f)$ prime ideals dividing $\f$ but not $\mathfrak{m}$ gives (i).
\end{proof}

For real quadratic fields we shall need the following simple criterion.

\begin{lemma}\label{lem:nottwist}
Let $K$ be real quadratic, $\chi$ a non-principal Dirichlet character modulo $\f$, and $\Hk$ a Hecke character on $P(\f)$ belonging to $\chi$ whose infinity part \eqref{eq:infpart-real} has $\eta_1=\eta_2=1$. Then $\Hk$ is not a norm twist.
\end{lemma}

\begin{proof}
Let $\beta\in\OK$ be coprime to $\f$ with $\chi(\beta)\neq1$, and let $f\in\OK$ be a generator of $\f$ if $\f$ is principal, or otherwise any non-zero element of $\f$. Choose $m\in\mathbb{N}$ so large that $|\sigma_i(fT^m)|>\max(1,|\sigma_i(\beta)|)$ for $i=1,2$, and put $\alpha_1=1+fT^m$ and $\alpha_2=\beta+fT^m$. Then $\alpha_1,\alpha_2\in K(\f)$, $\chi(\alpha_1)=1$ and $\chi(\alpha_2)=\chi(\beta)$. Since the absolute value is ultrametric, $\sigma_i(\alpha_1)$ and $\sigma_i(\alpha_2)$ both have the same degree and the same leading coefficient as $\sigma_i(fT^m)$. As $\eta_1=\eta_2=1$, the infinity part only depends on $\sgn\sigma_i(\cdot)$ and $\deg\sigma_i(\cdot)$, so $\chi_\infty(\alpha_1)=\chi_\infty(\alpha_2)$, and also $\N(\alpha_1)=\N(\alpha_2)$ by \eqref{eq:norm-sigma}. Hence
\[
\frac{\Hk((\alpha_2))}{\Hk((\alpha_1))}=\chi(\beta)\neq1=\frac{\N(\alpha_2)^{s_0}}{\N(\alpha_1)^{s_0}}
\]
for every $s_0\in\mathbb{C}$, so $\Hk$ is not a norm twist.
\end{proof}

\section{Dirichlet approximation}\label{sec:dirichlet}

A general version of Dirichlet's approximation theorem for function fields was established by Ganguly in the appendix to \cite{BaMo}, based on Thunder's adelic version of Minkowski's theorem \cite{Thu}; for real quadratic function fields of class number one, see also \cite[Theorem~2]{BaMo2}. These results provide, for each box size, an approximation $a/c$, but do not by themselves show that infinitely many distinct quotients $a/c$ arise. We therefore give a self-contained proof of the version we need. No assumption on the class number is required, since we never need the numerator and denominator to be coprime (cf.\ \eqref{eq:Dfa} below).

For $x=(x_1,x_2)\in\kinf^2$ and $\kappa\in K$ we write $\sigma(\kappa)x:=(\sigma_1(\kappa)x_1,\sigma_2(\kappa)x_2)$. Every $w\in\kinf$ can be written uniquely as $w=[w]+\{w\}$ with $[w]\in A$ and $\{w\}\in T^{-1}\Fq[[T^{-1}]]$, and both maps are $\Fq$-linear. For $m\ge0$ let $\mathrm{pr}_m(w)\in\Fq^m$ be the vector of coefficients of $T^{-1},\dots,T^{-m}$ in $\{w\}$. Thus $\mathrm{pr}_m(w)=0$ if and only if $|\{w\}|\le q^{-m-1}$.

\begin{lemma}[Box principle]\label{lem:box}
Let $K$ be real quadratic, $y\ge g+1$ and $x\in\kinf^2$. Then there are $c\in\OK\setminus\{0\}$ and $a\in\OK$ such that
\[
|\sigma_i(c)|\le q^y\quad\text{and}\quad|\sigma_i(c)x_i-\sigma_i(a)|\le q^{g-y}\qquad(i=1,2).
\]
In particular,
\begin{equation}\label{eq:box}
\Big|x_i-\frac{\sigma_i(a)}{\sigma_i(c)}\Big|\le\frac{q^g}{\N(c)}\qquad(i=1,2).
\end{equation}
\end{lemma}

\begin{proof}
For $z=(z_1,z_2)\in\kinf^2$ put $u(z)=(z_1+z_2)/2$ and $v(z)=(z_1-z_2)/(2\sqrt D)$. This is a $\kinf$-linear bijection of $\kinf^2$ with $u(\sigma(\alpha+\beta\sqrt D))=\alpha$ and $v(\sigma(\alpha+\beta\sqrt D))=\beta$ for $\alpha,\beta\in k$, so it maps $\sigma(\OK)$ onto $A^2$. Since $|2|=1$ and $|\sqrt D|=q^{g+1}$, and $z_{1,2}=u(z)\pm\sqrt D\,v(z)$, the ultrametric inequality gives
\begin{equation}\label{eq:uvnorm}
\max(|z_1|,|z_2|)=\max\big(|u(z)|,\,q^{g+1}|v(z)|\big)
\end{equation}
(as in the proof of Lemma~\ref{lem:Vy}). Consider the $\Fq$-linear map
\[
\Phi:V_y\to\Fq^{\,y-g-1}\times\Fq^{\,y},\qquad\kappa\mapsto\Big(\mathrm{pr}_{y-g-1}\big(u(\sigma(\kappa)x)\big),\,\mathrm{pr}_y\big(v(\sigma(\kappa)x)\big)\Big).
\]
By Lemma~\ref{lem:Vy}, $\dim_{\Fq}V_y=2y+1-g>2y-g-1$, so there is $c\in V_y\setminus\{0\}$ with $\Phi(c)=0$. Put $a=[u(\sigma(c)x)]+[v(\sigma(c)x)]\sqrt D\in\OK$ and $z=\sigma(c)x-\sigma(a)$. Then $u(z)=\{u(\sigma(c)x)\}$ and $v(z)=\{v(\sigma(c)x)\}$, so $|u(z)|\le q^{g-y}$ and $|v(z)|\le q^{-y-1}$. By \eqref{eq:uvnorm}, $|z_i|\le q^{g-y}$. Finally, as $c\in V_y$, we have $|\sigma_{3-i}(c)|\le q^y$, and by \eqref{eq:norm-sigma}
\[
\Big|x_i-\frac{\sigma_i(a)}{\sigma_i(c)}\Big|\le\frac{q^{g-y}}{|\sigma_i(c)|}\le\frac{q^g}{|\sigma_i(c)|\,|\sigma_{3-i}(c)|}=\frac{q^g}{\N(c)}.\qedhere
\]
\end{proof}

Lemma~\ref{lem:box} alone does not yield infinitely many \emph{distinct} quotients: for instance, if $x=(0,z)$ with $0<|z|\le q^g$, then $(a,c)=(0,1)$ satisfies \eqref{eq:box}. The following theorem removes this obstruction at the cost of a constant depending on $x$, which is harmless for our purposes.

\begin{theorem}\label{thm:dirichlet}
Let $K$ be a real quadratic extension of $k$ and $x\in\kinf^2\setminus\sigma(K)$. There is a constant $C=C(K,x)\ge1$ such that infinitely many distinct $r\in K$ can be written as $r=a/c$ with $(a,c)\in\OK\times(\OK\setminus\{0\})$ and
\begin{equation}\label{eq:dirichlet}
\Big|x_i-\frac{\sigma_i(a)}{\sigma_i(c)}\Big|\le\frac{C}{\N(c)}\qquad(i=1,2).
\end{equation}
\end{theorem}

\begin{proof}
\emph{Step 1: exceptional elements.} Since $\sigma_i$ is injective, for each $i\in\{1,2\}$ there is at most one $r_i\in K$ with $\sigma_i(r_i)=x_i$. As $x\notin\sigma(K)$, we have $\delta_i:=|x_{3-i}-\sigma_{3-i}(r_i)|>0$ whenever $r_i$ exists. Choose $m\ge0$ with $q^m\delta_i>q^g$ for each existing $r_i$, and put $x'=T^mx$. Then $x'\notin\sigma(K)$, and the only $s\in K$ with $\sigma_i(s)=x_i'$ is $T^mr_i$, for which $|x'_{3-i}-\sigma_{3-i}(T^mr_i)|=q^m\delta_i>q^g$.

\emph{Step 2: no quotient recurs infinitely often.} For each $y\ge g+1$ let $(a_y,c_y)$ be given by Lemma~\ref{lem:box} applied to $x'$, and put $s_y=a_y/c_y$. Fix $s\in K$ and suppose $s=s_y$. Put $\rho_i=|x_i'-\sigma_i(s)|$, so that $\rho_i|\sigma_i(c_y)|\le q^{g-y}$ and $|\sigma_i(c_y)|\le q^y$. Not both $\rho_i$ vanish, since $x'\notin\sigma(K)$.
\begin{itemize}
\item If $\rho_1\rho_2\neq0$, then $1\le\N(c_y)\le q^{2g-2y}/(\rho_1\rho_2)$, which bounds $y$.
\item If $\rho_1=0$, then $s=T^mr_1$ and $\rho_2>q^g$ by Step~1, so $1\le\N(c_y)=|\sigma_1(c_y)|\,|\sigma_2(c_y)|\le q^y\cdot q^{g-y}/\rho_2<1$, a contradiction. The case $\rho_2=0$ is symmetric.
\end{itemize}
Hence every $s\in K$ equals $s_y$ for only finitely many $y$, and the set $\{s_y:y\ge g+1\}$ is infinite.

\emph{Step 3: rescaling.} Put $r_y=a_y/(T^mc_y)$. Distinct $s_y$ give distinct $r_y$. Since $\N(T^mc_y)=q^{2m}\N(c_y)$, \eqref{eq:box} for $x'$ gives
\[
\Big|x_i-\frac{\sigma_i(a_y)}{\sigma_i(T^mc_y)}\Big|=q^{-m}\Big|x_i'-\frac{\sigma_i(a_y)}{\sigma_i(c_y)}\Big|\le\frac{q^{g-m}}{\N(c_y)}=\frac{q^{g+m}}{\N(T^mc_y)}.
\]
This is \eqref{eq:dirichlet} with $C=q^{g+m}$.
\end{proof}

\begin{remark}
We do not know whether $C$ can be chosen independently of $x$. This is not needed: in the proof of Theorem~\ref{thm:main}, $C$ enters only through the condition $q^{\varepsilon N/2}\ge C$ in Lemma~\ref{lem:approx}, and $x$ is fixed.
\end{remark}

If $h_K>1$, the elements $a$ and $c$ cannot in general be chosen coprime, since there is no greatest common divisor of elements. As in \cite{BaMo,BDM}, we work with the ideals
\begin{equation}\label{eq:Dfa}
\Dd:=(a)+(c)=\gcd\big((a),(c)\big),\qquad\f:=\Dd^{-1}(c),\qquad\aid:=\Dd^{-1}(a).
\end{equation}
Then $\f$ and $\aid$ are coprime integral ideals, and $(a/c)=\aid\f^{-1}$. We need the approximations with normalised denominators and with $\N(\f)$ tending to infinity.

\begin{lemma}\label{lem:sequence}
Let $K$ be real quadratic, $(x_1,x_2)\in\kinf^2\setminus\sigma(K)$, and let $C=C(K,x)$ be as in Theorem~\ref{thm:dirichlet}. Then there is a sequence $(a_n,c_n)_{n\ge1}$ in $\OK\times(\OK\setminus\{0\})$ such that, with $\f_n:=\gcd((a_n),(c_n))^{-1}(c_n)$,
\begin{enumerate}[label=(\roman*)]
\item $c_n$ is the normalised generator of $(c_n)$ in the sense of Lemma~\ref{lem:normgen};
\item $|x_i-\sigma_i(a_n)/\sigma_i(c_n)|\le C/\N(c_n)\le C/\N(\f_n)$ for $i=1,2$;
\item $\N(\f_n)\to\infty$ as $n\to\infty$.
\end{enumerate}
\end{lemma}

\begin{proof}
Let $r_1,r_2,\dots$ be the infinitely many distinct elements of $K$ provided by Theorem~\ref{thm:dirichlet}, $r_n=a_n/c_n$. Multiplying $a_n$ and $c_n$ by the same unit, we may assume (i); this changes neither $r_n$ nor $\f_n$ nor \eqref{eq:dirichlet}. Since $\f_n\mid(c_n)$, we have $\N(\f_n)\le\N(c_n)$, which gives (ii). Suppose that $\N(\f_n)\le B$ for infinitely many $n$. Since $(r_n)=\aid_n\f_n^{-1}$ with $\aid_n$ integral, we have $r_n\in\f_n^{-1}\subseteq\mathfrak{c}^{-1}$, where $\mathfrak{c}$ is the product of the finitely many integral ideals of norm at most $B$. Moreover, $|\sigma_i(r_n)|\le\max(|x_i|,C)$ for $i=1,2$. The image of $\mathfrak{c}^{-1}$ under $\sigma$ is discrete in $\kinf^2$, so only finitely many elements of $\mathfrak{c}^{-1}$ satisfy these bounds. This contradicts the distinctness of the $r_n$ and proves (iii).
\end{proof}

\section{Counting lemmas}\label{sec:counting}

Throughout this section, $K=k(\sqrt D)$ is real quadratic with $D$ monic of degree $2g+2$, and $h_K=\#\Cl(\OK)$ is its (ideal) class number. Thanks to the ultrametric structure of $\kinf$, the lattice point counts we need are exact or elementary.

Let $Z_K(u)=L_K(u)/((1-u)(1-qu))$ be the zeta function of $K$, where $L_K\in\Z[u]$ has degree $2g$ and $L_K(1)=\tilde h_K=\#\Pic^0(K)$.

\begin{lemma}\label{lem:idealcount}
There is a constant $c_K=L_K(q^{-1})>0$ such that for all integers $M\ge0$,
\[
\#\{\aid\subseteq\OK:\aid\neq0,\ \N(\aid)\le q^M\}=c_Kq^M+O_K(1).
\]
Moreover, $c_K=q^{-g}\tilde h_K=Rh_Kq^{-g}$.
\end{lemma}

\begin{proof}
Since $\infty_1$ and $\infty_2$ have degree one, the generating series of the number $a_n$ of integral ideals of norm $q^n$ is
\[
\sum_{n\ge0}a_nu^n=Z_K(u)(1-u)^2=\frac{L_K(u)(1-u)}{1-qu}.
\]
Writing $L_K(u)(1-u)=\sum_{i=0}^{2g+1}p_iu^i$, we get $a_n=\sum_{i\le\min(n,2g+1)}p_iq^{n-i}$, so $a_n=q^n(1-q^{-1})L_K(q^{-1})$ for $n\ge2g+1$. Summing over $n\le M$ gives the claim with $c_K=L_K(q^{-1})$. By the functional equation $L_K(u)=q^gu^{2g}L_K(1/(qu))$, we have $L_K(q^{-1})=q^{-g}\tilde h_K$, and $\tilde h_K=Rh_K$ by Remark~\ref{rem:twoclass}.
\end{proof}

We regard characters $\lambda$ of the class group $\Cl(\OK)=I/P$ as functions on $I$. They are Hecke characters modulo $(1)$ of trivial infinity type; in particular they are unramified at $\infty_1,\infty_2$ with $\lambda(\infty_1)=\lambda(\infty_2)=1$ (Definition~\ref{def:infvalue}).

\begin{lemma}\label{lem:classchar}
Let $\lambda\neq1$ be a character of $\Cl(\OK)$. Then:
\begin{enumerate}[label=(\roman*)]
\item $\lambda$ is not a norm twist;
\item $\sum_{\N(\mathfrak{b})\le q^m}\lambda(\mathfrak{b})=O_K(1)$ for all $m\ge0$, the sum running over non-zero integral ideals.
\end{enumerate}
\end{lemma}

\begin{proof}
(i) By the approximation theorem there is $\alpha\in K^\times$ with $v_{\infty_1}(\alpha)=0$ and $v_{\infty_2}(\alpha)=-1$, so $\N((\alpha))=q$ by Lemma~\ref{lem:prodformula}. If $\lambda=\N(\cdot)^{s_0}$, then $q^{s_0}=\lambda((\alpha))=1$, and since all norms are powers of $q$, $\lambda$ would be trivial. (ii) By (i) and Theorem~\ref{thm:FE}, $\Lambda(s,\lambda)$ is a polynomial of degree $2g-2$ in $u=q^{-s}$, and $\sum_{n\ge0}\big(\sum_{\N(\mathfrak{b})=q^n}\lambda(\mathfrak{b})\big)u^n=\Lambda(s,\lambda)(1-u)^2$ is a polynomial of degree $2g$. Hence the partial sums are bounded by $\sum_{n\le2g}a_n$, with $a_n$ as in the proof of Lemma~\ref{lem:idealcount}.
\end{proof}

\begin{lemma}\label{lem:classcount}
Let $\f$ be a non-zero integral ideal and $C\in\Cl(\OK)$ an ideal class. For all integers $M\ge0$,
\[
\#\{\mathfrak{k}\in C:\mathfrak{k}\subseteq\OK,\ \N(\mathfrak{k})\le q^M,\ \mathfrak{k}+\f=\OK\}=\frac{c_K}{h_K}q^M\frac{\varphi(\f)}{\N(\f)}+O_K\big(2^{\omega(\f)}\big).
\]
\end{lemma}

\begin{proof}
By orthogonality of the characters of $\Cl(\OK)$, the left-hand side equals
\[
\frac1{h_K}\sum_\lambda\overline{\lambda(C)}\sum_{\substack{\N(\mathfrak{k})\le q^M\\\mathfrak{k}+\f=\OK}}\lambda(\mathfrak{k})=\frac1{h_K}\sum_\lambda\overline{\lambda(C)}\sum_{\mathfrak{d}\mid\f}\mu(\mathfrak{d})\lambda(\mathfrak{d})\sum_{\N(\mathfrak{b})\le q^M/\N(\mathfrak{d})}\lambda(\mathfrak{b}).
\]
For $\lambda\neq1$ the innermost sum is $O(1)$ by Lemma~\ref{lem:classchar}(ii) (it is empty if $\N(\mathfrak{d})>q^M$). For $\lambda=1$ it equals $c_Kq^M/\N(\mathfrak{d})+O(1)$ by Lemma~\ref{lem:idealcount}; this also holds if $\N(\mathfrak{d})>q^M$, since then the sum vanishes and $c_Kq^M/\N(\mathfrak{d})\le c_K$. Summing over the $2^{\omega(\f)}$ square-free divisors $\mathfrak{d}$ of $\f$ gives the claim.
\end{proof}

\begin{lemma}\label{lem:DVy}
Let $\Dd$ be a non-zero integral ideal and $y\in\Z$ with $2y\ge\log_q\N(\Dd)$. Then
\[
\#\big((\Dd\cap V_y)\setminus\{0\}\big)\ll_K\frac{q^{2y}}{\N(\Dd)},
\]
with $V_y$ as in Lemma~\ref{lem:Vy}.
\end{lemma}

\begin{proof}
Put $E=\log_q\N(\Dd)$. A non-zero $\kappa\in\Dd\cap V_y$ generates an ideal $(\kappa)=\Dd\mathfrak{k}$ with $\mathfrak{k}$ integral, and $\N(\Dd\mathfrak{k})=|\sigma_1(\kappa)||\sigma_2(\kappa)|\le q^{2y}$, so $\N(\mathfrak{k})=q^m$ with $m\le2y-E$. Conversely, given the principal ideal $(\kappa)$, its generators are $c\epsilon^j\kappa_0$ with $c\in\Fq^\times$ and $j\in\Z$, where $\kappa_0$ is the normalised generator. If $d_i=\deg\sigma_i(\kappa_0)$, then $\deg\sigma_1(c\epsilon^j\kappa_0)=d_1+Rj$ and $\deg\sigma_2(c\epsilon^j\kappa_0)=d_2-Rj$, and both have to be $\le y$. Since $d_1+d_2=m+E$, the number of admissible $j$ is at most $(2y-m-E)/R+1$. Using the bound $a_m\ll q^m$ from Lemma~\ref{lem:idealcount}, we obtain
\[
\#\big((\Dd\cap V_y)\setminus\{0\}\big)\le(q-1)\sum_{0\le m\le2y-E}a_m\Big(\frac{2y-E-m}{R}+1\Big)\ll q^{2y-E}\sum_{l\ge0}q^{-l}(l+1)\ll q^{2y-E}.\qedhere
\]
\end{proof}

\section{Proof of Theorem~\ref{thm:main}}\label{sec:proof}

Throughout this section, $K=k(\sqrt D)$ is real quadratic with ideal class number $h_K$, $D$ is monic of degree $2g+2$, and $\epsilon,R$ are as in \eqref{eq:eps-norm}. We fix $x=(x_1,x_2)\in\kinf^2\setminus\sigma(K)$ and let $C=C(K,x)$ be as in Theorem~\ref{thm:dirichlet}. We write $\e(x)=\exp(2\pi ix)$ for $x\in\R$. All implied constants may depend on $K$ and $\varepsilon$; the threshold in ``$F$ large enough'' may depend in addition on $x$.

\subsection{Choice of parameters}
Fix $0<\varepsilon<1/10$. Let $(a,c)=(a_n,c_n)$ be a member of the sequence in Lemma~\ref{lem:sequence}, and let $\Dd,\f,\aid$ be as in \eqref{eq:Dfa}. Put
\[
F:=\log_q\N(\f),\qquad E:=\log_q\N(\Dd),\qquad\text{so that }\N(c)=q^{F+E}.
\]
By Lemma~\ref{lem:sequence}(iii) we may assume $F$ to be as large as we like. We define
\begin{equation}\label{eq:params}
N:=\Big\lfloor\frac{F}{2/3-\varepsilon/2}\Big\rfloor,\qquad L:=N-F,\qquad y:=\Big\lfloor\frac{L+E}{2}\Big\rfloor-R,\qquad M:=2y-R-E.
\end{equation}
Then $(2/3-\varepsilon/2)N\le F<(2/3-\varepsilon/2)(N+1)$, so that for $F$ large enough
\begin{equation}\label{eq:params-ineq}
\begin{gathered}
N>F,\qquad\Big(\frac13+\frac\varepsilon2\Big)N-1<L\le\Big(\frac13+\frac\varepsilon2\Big)N,\\
L+E-1-2R\le2y\le L+E-2R<F+E,\qquad M\ge L-1-3R\ge0,\qquad y\ge g+1.
\end{gathered}
\end{equation}
(For $L<F$ we use $\varepsilon<1/3$.) Note that the size of $\Dd$ is not controlled; it will cancel out in the final estimate.

\subsection{Reduction to prime elements in residue classes}
Let $\mathcal{P}_N$ be the set of normalised generators (Lemma~\ref{lem:normgen}) of the principal prime ideals $\p$ of $\OK$ with $\N(\p)=q^N$. Since $N>F$, no such $\p$ divides $\f$, so $\mathcal{P}_N\subseteq K(\f)$. Let
\[
\cS:=\{\kappa\in\Dd\cap V_y:\Dd^{-1}(\kappa)+\f=\OK\},
\]
with $V_y$ as in Lemma~\ref{lem:Vy}. For $\kappa\in\cS$ we have $v_\p(\kappa)=v_\p(\Dd)=v_\p(a)$ for all $\p\mid\f$ (recall that $\aid=\Dd^{-1}(a)$ is coprime to $\f$), so $a\kappa^{-1}\in K(\f)$.

\begin{lemma}\label{lem:congruence}
Let $p\in\mathcal{P}_N$ and $\kappa\in\cS$. Then $\kappa\equiv ap\bmod c$ if and only if $ap\kappa^{-1}\eqs1\bmod\f$.
\end{lemma}

\begin{proof}
Since $a,\kappa\in\Dd$, the ideal $\mathfrak{n}:=\Dd^{-1}(\kappa-ap)$ is integral (or zero). As $(c)=\Dd\f$, we have $c\mid\kappa-ap$ if and only if $\f\mid\mathfrak{n}$, i.e.\ $v_\p(\kappa-ap)-v_\p(\Dd)\ge v_\p(\f)$ for all $\p\mid\f$. Since $v_\p(\Dd)=v_\p(\kappa)$ for $\p\mid\f$, this is equivalent to $v_\p(1-ap\kappa^{-1})\ge v_\p(\f)$ for all $\p\mid\f$, i.e.\ to $ap\kappa^{-1}\eqs1\bmod\f$.
\end{proof}

\begin{lemma}\label{lem:approx}
Let $p\in\mathcal{P}_N$ and $\kappa\in\cS$ with $ap\kappa^{-1}\eqs1\bmod\f$. Then $b:=(ap-\kappa)/c\in\OK$ and
\[
\Big|x_i-\frac{\sigma_i(b)}{\sigma_i(p)}\Big|\le Cq^{-F}\le\N(\p)^{-2/3+\varepsilon}\qquad(i=1,2),
\]
provided that $F$ is large enough (in terms of $K$, $\varepsilon$ and $x$).
\end{lemma}

\begin{proof}
By Lemma~\ref{lem:congruence}, $b\in\OK$, and $ap-bc=\kappa$. By Lemma~\ref{lem:normgen}, $|\sigma_i(c)|\ge q^{(F+E-R)/2}$ and $|\sigma_i(p)|\ge q^{(N-R)/2}$. Hence
\[
\Big|\frac{\sigma_i(a)}{\sigma_i(c)}-\frac{\sigma_i(b)}{\sigma_i(p)}\Big|=\frac{|\sigma_i(\kappa)|}{|\sigma_i(c)|\,|\sigma_i(p)|}\le q^{y+R-(N+F+E)/2}\le q^{-F}
\]
by the definition of $y$. Since $|x_i-\sigma_i(a)/\sigma_i(c)|\le C/\N(c)\le Cq^{-F}$ and $C\ge1$, the ultrametric inequality gives $|x_i-\sigma_i(b)/\sigma_i(p)|\le Cq^{-F}$. Finally, $Cq^{-F}\le Cq^{-(2/3-\varepsilon/2)N}\le q^{-(2/3-\varepsilon)N}=\N(\p)^{-2/3+\varepsilon}$ as soon as $q^{\varepsilon N/2}\ge C$.
\end{proof}

Consequently, Theorem~\ref{thm:main} follows once we show that
\begin{equation}\label{eq:Sigma}
\Sigma:=\#\{(p,\kappa)\in\mathcal{P}_N\times\cS:ap\kappa^{-1}\eqs1\bmod\f\}>0
\end{equation}
for all sufficiently large $F$. Indeed, $F$ and hence $N=\log_q\N(\p)$ tend to infinity along our sequence, so we obtain infinitely many pairs $(b,\pi)=(b,p)$ satisfying \eqref{eq:main}.

Since $a\kappa^{-1}$ and $p$ lie in $K(\f)$, orthogonality of the Dirichlet characters modulo $\f$ gives
\begin{equation}\label{eq:orth}
\Sigma=\frac1{\varphi(\f)}\sum_{\chi\bmod\f}\sum_{\kappa\in\cS}\chi(a\kappa^{-1})\sum_{p\in\mathcal{P}_N}\chi(p).
\end{equation}

\subsection{Exact decomposition into Hecke characters}
Since $\OK^\times$ is infinite, $\chi(p)$ is not a function of the ideal $(p)$, and we cannot directly apply the Riemann hypothesis. The following proposition is the key step. Recall the Hecke character $\psi$ of order $2R$ from \eqref{eq:psi}. For $\gamma\in\R$ and $n\in\Z$ put
\begin{equation}\label{eq:cn}
c_n(\gamma):=\frac1{2R}\sum_{j=-R+1}^{R}\e\Big(-\frac{(\gamma+n)j}{2R}\Big).
\end{equation}

\begin{proposition}\label{prop:decomp}
Let $\chi$ be a Dirichlet character modulo $\f$, and let $\gamma_\chi\in(-1/2,1/2]$ be defined by $\chi(\epsilon)=\e(-\gamma_\chi)$. For $\alpha\in K(\f)$ put
\begin{equation}\label{eq:chiHecke}
\Hk^{(\chi)}((\alpha)):=\chi(\alpha)\,\overline{\chi\big(\sgn\sigma_1(\alpha)\big)}\,\Big|\frac{\sigma_1(\alpha)}{\sigma_2(\alpha)}\Big|^{i\pi\gamma_\chi/(R\log q)}.
\end{equation}
Then the following hold.
\begin{enumerate}[label=(\roman*)]
\item $\Hk^{(\chi)}$ is a well-defined Hecke character on $P(\f)$ belonging to $\chi$. Each of its $h_K$ extensions to $I(\f)$ (Proposition~\ref{prop:extension}) is a Hecke character modulo $\f$ whose conductor has degree at most $F+1$.
\item If $\chi=\chi_0$ is principal, then $\gamma_\chi=0$ and $\Hk^{(\chi)}$ is trivial on $P(\f)$.
\item If $\chi\neq\chi_0$, then for every $n\in\Z$, no Hecke character on $I(\f)$ whose restriction to $P(\f)$ equals $\Hk^{(\chi)}\psi^n$ is a norm twist.
\item For every normalised generator $p$ of a principal prime ideal not dividing $\f$,
\begin{equation}\label{eq:exact}
\chi(p)=\sum_{n=0}^{2R-1}c_n(\gamma_\chi)\,\big(\Hk^{(\chi)}\psi^n\big)((p)).
\end{equation}
\item $\sum_{n=0}^{2R-1}|c_n(\gamma)|^2=1$, hence $\sum_{n=0}^{2R-1}|c_n(\gamma)|\le\sqrt{2R}$; moreover $c_n(0)=1$ if $n\equiv0\bmod2R$ and $c_n(0)=0$ otherwise.
\end{enumerate}
\end{proposition}

\begin{proof}
(i) The infinity part of \eqref{eq:chiHecke} is \eqref{eq:infpart-real} with $\omega_1=\overline{\chi|_{\Fq^\times}}$, $\omega_2=1$, $\eta_1=\eta_2=1$ and $v=\pi\gamma_\chi/(R\log q)$. This is the case $n=0$ of Theorem~\ref{thm:hecke-real}, since the condition at constants gives $\omega_2=\overline{\chi\omega_1}=1$, and \eqref{eq:gamma} reads $\chi(\epsilon)=\e(-\gamma)$ because $\omega_2=1$ and $\eta_1=\eta_2=1$. For the reader's convenience, we verify directly that \eqref{eq:chiHecke} is trivial on units. For $c\in\Fq^\times$ we get $\chi(c)\overline{\chi(c)}\cdot1=1$. For $\epsilon$, using \eqref{eq:eps-norm} and \eqref{eq:eps-sigma2}, we get
\[
\chi(\epsilon)\cdot\overline{\chi(1)}\cdot q^{2R\cdot i\pi\gamma_\chi/(R\log q)}=\e(-\gamma_\chi)\e(\gamma_\chi)=1.
\]
Hence $\Hk^{(\chi)}$ is well defined and multiplicative on $P(\f)$. By Propositions~\ref{prop:extension} and~\ref{prop:idelic}, its extensions to $I(\f)$ are Hecke characters, with conductor dividing $\f\cdot\infty_1$, as the infinity part is at most tamely ramified at $\infty_1$ and unramified at $\infty_2$; so the degree of the conductor is at most $F+1$.

(ii) If $\chi=\chi_0$, then $\chi(\epsilon)=1$, so $\gamma_\chi=0$, and $\chi(\sgn\sigma_1(\alpha))=1$.

(iii) The character $\Hk^{(\chi)}\psi^n$ belongs to $\chi$ (as $\psi$ has trivial finite part), and its infinity part has trivial principal-unit components. By Lemma~\ref{lem:nottwist}, it is not a norm twist on $P(\f)$, so no extension of it to $I(\f)$ is a norm twist either.

(iv) Let $j=\deg\sigma_1(p)-\deg\sigma_2(p)$, so that $j\in J:=\{-R+1,\dots,R\}$ by \eqref{eq:normgen}. Since $\sgn\sigma_1(p)=1$,
\[
\Hk^{(\chi)}((p))=\chi(p)\,q^{ij\pi\gamma_\chi/(R\log q)}=\chi(p)\,\e\Big(\frac{\gamma_\chi j}{2R}\Big),\qquad\psi^n((p))=\e\Big(\frac{nj}{2R}\Big).
\]
The set $J$ is a complete system of residues modulo $2R$. By discrete Fourier inversion on $\Z/2R\Z$, the function $j\mapsto\e(-\gamma_\chi j/(2R))$ on $J$ satisfies
\[
\e\Big(-\frac{\gamma_\chi j}{2R}\Big)=\sum_{n=0}^{2R-1}c_n(\gamma_\chi)\,\e\Big(\frac{nj}{2R}\Big)\qquad(j\in J),
\]
with $c_n$ as in \eqref{eq:cn}. Multiplying by $\Hk^{(\chi)}((p))$ gives \eqref{eq:exact}.

(v) This is Parseval's identity $\sum_n|c_n(\gamma)|^2=\frac1{2R}\sum_{j\in J}|\e(-\gamma j/(2R))|^2=1$, followed by the Cauchy--Schwarz inequality. The last assertion is the orthogonality of additive characters of $\Z/2R\Z$.
\end{proof}

\begin{remark}
For $\gamma\notin\Z$, summing the geometric series gives
\[
c_n(\gamma)=(-1)^n\frac{\sin(\pi\gamma)}{2R\sin\big(\pi(\gamma+n)/(2R)\big)}\,\e\Big(-\frac{\gamma+n}{4R}\Big),
\]
which, as $R\to\infty$ and with $n$ fixed, tends to $(-1)^n\sin(\pi\gamma)/(\pi(\gamma+n))$. This is the Fourier coefficient that appears for real quadratic number fields in \cite[\S4]{BDM}, where $\theta=-\gamma$. In the number field case, however, the Fourier series has infinitely many terms and has to be truncated, which produces an error term whose control requires information on the spacing of the numbers $\log|\sigma_1(p)/\sigma_2(p)|$. In the function field case, $\deg\sigma_1(p)-\deg\sigma_2(p)$ takes only $2R$ values on normalised generators, and \eqref{eq:exact} is an exact identity with a bounded number of terms. This reflects Corollary~\ref{cor:finite}.
\end{remark}

\subsection{Sums over prime elements}
For each Dirichlet character $\chi$ modulo $\f$ and $0\le n<2R$, fix an extension $\Xi_{\chi,n}$ of $\Hk^{(\chi)}\psi^n$ from $P(\f)$ to $I(\f)$; for $\chi=\chi_0$ we take $\Xi_{\chi_0,0}$ to be the principal character modulo $\f$. Since $\Xi_{\chi,n}$ and $\Hk^{(\chi)}\psi^n$ agree on principal ideals, Proposition~\ref{prop:decomp}(iv) and the orthogonality relation for the characters $\lambda$ of $\Cl(\OK)$ give
\begin{equation}\label{eq:primesum}
\sum_{p\in\mathcal{P}_N}\chi(p)=\sum_{n=0}^{2R-1}c_n(\gamma_\chi)\sum_{\substack{\p\text{ principal}\\\N(\p)=q^N}}\Xi_{\chi,n}(\p)=\sum_{n=0}^{2R-1}\frac{c_n(\gamma_\chi)}{h_K}\sum_\lambda\sum_{\N(\p)=q^N}\big(\Xi_{\chi,n}\lambda\big)(\p).
\end{equation}
If $\chi=\chi_0$, then only $n=0$ occurs by Proposition~\ref{prop:decomp}(ii) and (v). The term $\lambda=1$ is evaluated by Lemma~\ref{lem:PNT}(ii), and the terms $\lambda\neq1$ are bounded by Lemma~\ref{lem:PNT}(i), which applies by Lemma~\ref{lem:classchar}(i). We get
\begin{equation}\label{eq:principal-primes}
\sum_{p\in\mathcal{P}_N}\chi_0(p)=\#\{\p\text{ principal}:\N(\p)=q^N\}=\frac{q^N}{h_KN}+O\Big(\frac{q^{N/2}}{N}\Big).
\end{equation}
If $\chi\neq\chi_0$, then $\Xi_{\chi,n}\lambda$ is a Hecke character modulo $\f$ with conductor of degree at most $F+1$ whose restriction to $P(\f)$ is $\Hk^{(\chi)}\psi^n$, so it is not a norm twist by Proposition~\ref{prop:decomp}(iii). Lemma~\ref{lem:PNT}(i) and Proposition~\ref{prop:decomp}(v) give
\begin{equation}\label{eq:nonprincipal-primes}
\Big|\sum_{p\in\mathcal{P}_N}\chi(p)\Big|\le\sqrt{2R}\max_{n,\lambda}\Big|\sum_{\N(\p)=q^N}\big(\Xi_{\chi,n}\lambda\big)(\p)\Big|\ll F\frac{q^{N/2}}{N},
\end{equation}
where we used $\omega(\f)\le F$.

\subsection{Sums over $\kappa$}

\begin{lemma}\label{lem:kappa}
\begin{enumerate}[label=(\roman*)]
\item If $\kappa_1,\kappa_2\in\cS$ and $\kappa_2\kappa_1^{-1}\eqs1\bmod\f$, then $\kappa_1=\kappa_2$. Consequently,
\[
\sum_{\chi\bmod\f}\Big|\sum_{\kappa\in\cS}\chi(a\kappa^{-1})\Big|\le\varphi(\f)\,(\#\cS)^{1/2}.
\]
\item We have $\#\cS\ll q^{2y-E}$ and
\[
\#\cS\ge\frac{c_K}{h_K}q^M\frac{\varphi(\f)}{\N(\f)}-O\big(2^{\omega(\f)}\big).
\]
\end{enumerate}
\end{lemma}

\begin{proof}
(i) As in the proof of Lemma~\ref{lem:congruence}, since $v_\p(\kappa_1)=v_\p(\Dd)$ for $\p\mid\f$, the condition $\kappa_2\kappa_1^{-1}\eqs1\bmod\f$ is equivalent to $\f\mid\Dd^{-1}(\kappa_2-\kappa_1)$, i.e.\ to $\Dd\f\mid(\kappa_2-\kappa_1)$. If $\kappa_1\neq\kappa_2$, this gives
\[
q^{F+E}=\N(\Dd\f)\le\N(\kappa_2-\kappa_1)=|\sigma_1(\kappa_2-\kappa_1)|\,|\sigma_2(\kappa_2-\kappa_1)|\le q^{2y},
\]
contradicting $2y<F+E$. For the second statement, note that $\chi(a\kappa_1^{-1})\overline{\chi(a\kappa_2^{-1})}=\chi(\kappa_2\kappa_1^{-1})$, where $\kappa_2\kappa_1^{-1}\in K(\f)$. By the Cauchy--Schwarz inequality and orthogonality,
\[
\sum_\chi\Big|\sum_{\kappa\in\cS}\chi(a\kappa^{-1})\Big|\le\varphi(\f)^{1/2}\Big(\varphi(\f)\sum_{\substack{\kappa_1,\kappa_2\in\cS\\\kappa_2\kappa_1^{-1}\eqs1\bmod\f}}1\Big)^{1/2}=\varphi(\f)(\#\cS)^{1/2}.
\]
(ii) The upper bound follows from $\cS\subseteq(\Dd\cap V_y)\setminus\{0\}$ and Lemma~\ref{lem:DVy}, noting $2y\ge E$ by \eqref{eq:params-ineq}. For the lower bound, let $\mathfrak{k}$ be an integral ideal in the class of $\Dd^{-1}$, coprime to $\f$, with $\N(\mathfrak{k})\le q^M$. Then $\Dd\mathfrak{k}$ is principal; let $\kappa$ be its normalised generator. By \eqref{eq:normgen-size}, $|\sigma_i(\kappa)|\le q^{(E+M+R)/2}=q^y$, and $\kappa\in\Dd$ with $\Dd^{-1}(\kappa)=\mathfrak{k}$ coprime to $\f$. Hence $\kappa\in\cS$, and distinct $\mathfrak{k}$ give distinct $\kappa$. The claim now follows from Lemma~\ref{lem:classcount}.
\end{proof}

\subsection{Conclusion}
Inserting \eqref{eq:principal-primes}, \eqref{eq:nonprincipal-primes} and Lemma~\ref{lem:kappa} into \eqref{eq:orth}, we get
\begin{equation}\label{eq:Sigma-lower}
\Sigma\ge\frac{\#\cS}{\varphi(\f)}\Big(\frac{q^N}{h_KN}-O\Big(\frac{q^{N/2}}{N}\Big)\Big)-O\Big((\#\cS)^{1/2}F\frac{q^{N/2}}{N}\Big).
\end{equation}
We now estimate the three terms, using \eqref{eq:params-ineq} and Lemma~\ref{lem:arith} in the form $2^{\omega(\f)}\ll q^{\varepsilon F/8}$ and $\varphi(\f)\gg q^{(1-\varepsilon/8)F}$.

\emph{Main term.} By Lemma~\ref{lem:kappa}(ii) and $M-F+N\ge2L-1-3R$,
\[
\frac{\#\cS}{\varphi(\f)}\cdot\frac{q^N}{h_KN}\ge\frac{c_K}{h_K^2}\cdot\frac{q^{M-F+N}}{N}-O\Big(2^{\omega(\f)}\frac{q^N}{\varphi(\f)N}\Big)\ge\frac{c_K}{h_K^2}q^{-1-3R}\frac{q^{2L}}{N}-O\big(q^{L+\varepsilon F/4}\big).
\]
Here $q^{2L}\ge q^{(2/3+\varepsilon)N-2}$, whereas $q^{L+\varepsilon F/4}\le q^{(1/3+\varepsilon)N}$.

\emph{Error from the principal character.} By Lemma~\ref{lem:kappa}(ii) and $2y-E\le L$,
\[
\frac{\#\cS}{\varphi(\f)}\cdot\frac{q^{N/2}}{N}\ll q^{2y-E-F+\varepsilon F/8+N/2}\le q^{N/2+N-2F+\varepsilon N}\le q^{(1/6+2\varepsilon)N},
\]
since $N-2F\le(-1/3+\varepsilon)N$.

\emph{Error from the non-principal characters.} By Lemma~\ref{lem:kappa}(ii),
\[
(\#\cS)^{1/2}F\frac{q^{N/2}}{N}\ll F\frac{q^{y-E/2+N/2}}{N}\le F\frac{q^{L/2+N/2}}{N}\le F\frac{q^{(2/3+\varepsilon/4)N}}{N}.
\]
In all three estimates the quantity $E=\log_q\N(\Dd)$ has cancelled out. Altogether,
\[
\Sigma\ge\frac{c_K}{h_K^2}q^{-3-3R}\frac{q^{(2/3+\varepsilon)N}}{N}-O\Big(F\frac{q^{(2/3+\varepsilon/4)N}}{N}+q^{(1/3+\varepsilon)N}+q^{(1/6+2\varepsilon)N}\Big),
\]
which is positive for all sufficiently large $F$, since $F\le N$ and $\varepsilon<1/10$. This proves \eqref{eq:Sigma} and hence Theorem~\ref{thm:main}.\qed

\subsection*{Acknowledgements}
The author thanks Stephan Baier for helpful discussions. The author also thanks Will Sawin for a helpful comment on Hecke characters at the places at infinity. This work was started at the Indian Statistical Institute, Kolkata, and completed at the Indian Institute of Technology Bombay; the author thanks both institutes for excellent working environments.


\begin{thebibliography}{99}

\bibitem{Bai} S. Baier, \emph{Diophantine approximation on lines in $\mathbb{C}^2$ with Gaussian prime constraints}, Eur. J. Math. \textbf{3} (2017), no.~3, 614--649.

\bibitem{BDM} S. Baier, S. Das and E. A. Molla, \emph{Diophantine approximation with prime denominator in quadratic number fields under GRH}, Ramanujan J. \textbf{65} (2024), 1363--1405.

\bibitem{BaMa} S. Baier and D. Mazumder, \emph{Diophantine approximation with prime restriction in real quadratic number fields}, Math. Z. \textbf{299} (2021), no.~1--2, 699--750.

\bibitem{BMT} S. Baier, D. Mazumder and M. Technau, \emph{On the distribution of $\alpha p$ modulo one in quadratic number fields}, Unif. Distrib. Theory \textbf{16} (2021), no.~2, 1--48.

\bibitem{BaMo} S. Baier, E. A. Molla and A. Ganguly, \emph{Diophantine approximation with prime restriction in function fields}, J. Number Theory \textbf{241} (2022), 57--90.

\bibitem{BaMo2} S. Baier and E. A. Molla, \emph{Diophantine approximation with prime denominator in real quadratic function fields}, Finite Fields Appl. \textbf{91} (2023), Paper No.~102242, 40~pp.

\bibitem{BaTe} S. Baier and M. Technau, \emph{On the distribution of $\alpha p$ modulo one in imaginary quadratic number fields with class number one}, J. Th\'eor. Nombres Bordeaux \textbf{32} (2020), no.~3, 719--760.

\bibitem{Har83} G. Harman, \emph{On the distribution of $\alpha p$ modulo one}, J. London Math. Soc. (2) \textbf{27} (1983), no.~1, 9--18.

\bibitem{Har96} G. Harman, \emph{On the distribution of $\alpha p$ modulo one. II}, Proc. London Math. Soc. (3) \textbf{72} (1996), no.~2, 241--260.

\bibitem{Har19} G. Harman, \emph{Diophantine approximation with Gaussian primes}, Q. J. Math. \textbf{70} (2019), no.~4, 1505--1519.

\bibitem{Hay} D. R. Hayes, \emph{The distribution of irreducibles in GF$[q,x]$}, Trans. Amer. Math. Soc. \textbf{117} (1965), 101--127.

\bibitem{HBJ} D. R. Heath-Brown and C. Jia, \emph{The distribution of $\alpha p$ modulo one}, Proc. London Math. Soc. (3) \textbf{84} (2002), no.~1, 79--104.

\bibitem{Jia1} C. Jia, \emph{On the distribution of $\alpha p$ modulo one}, J. Number Theory \textbf{45} (1993), no.~3, 241--253.

\bibitem{Jia2} C. Jia, \emph{On the distribution of $\alpha p$ modulo one. II}, Sci. China Ser. A \textbf{43} (2000), no.~7, 703--721.

\bibitem{Mat} K. Matom\"aki, \emph{The distribution of $\alpha p$ modulo one}, Math. Proc. Cambridge Philos. Soc. \textbf{147} (2009), no.~2, 267--283.

\bibitem{Miy} T. Miyake, \emph{Modular Forms}, Springer Monographs in Mathematics, Springer, Berlin, 2006.

\bibitem{Ros} M. Rosen, \emph{Number Theory in Function Fields}, Graduate Texts in Mathematics \textbf{210}, Springer, New York, 2002.

\bibitem{Tat} J. T. Tate, \emph{Fourier analysis in number fields and Hecke's zeta-functions}, in: Algebraic Number Theory (J. W. S. Cassels and A. Fr\"ohlich, eds.), Academic Press, London, 1967, 305--347.

\bibitem{Thu} J. L. Thunder, \emph{An adelic Minkowski--Hlawka theorem and an application to Siegel's lemma}, J. Reine Angew. Math. \textbf{475} (1996), 167--185.

\bibitem{Vau} R. C. Vaughan, \emph{On the distribution of $\alpha p$ modulo 1}, Mathematika \textbf{24} (1977), no.~2, 135--141.

\bibitem{Vin} I. M. Vinogradov, \emph{The Method of Trigonometrical Sums in the Theory of Numbers}, translated, revised and annotated by K. F. Roth and A. Davenport, Dover, Mineola, NY, 2004.

\bibitem{Wei48} A. Weil, \emph{Sur les courbes alg\'ebriques et les vari\'et\'es qui s'en d\'eduisent}, Actualit\'es Sci. Ind. \textbf{1041}, Hermann, Paris, 1948.

\bibitem{Wei74} A. Weil, \emph{Basic Number Theory}, 3rd ed., Grundlehren der Mathematischen Wissenschaften \textbf{144}, Springer, Berlin, 1974.

\end{thebibliography}
\end{document}